\documentclass[11pt]{amsart}
\usepackage{amssymb,amsthm,amsmath,amstext}
\usepackage{mathrsfs}  % allows nice script font for sheaves
\usepackage{bm}        % allows bold italic letters in math mode
\usepackage{mathtools} % allows more extendible arrows
\usepackage{color}
\usepackage[bbgreekl]{mathbbol}

\usepackage{cite}

\usepackage[left=1.5in,top=1in,right=1.5in,bottom=1in]{geometry}

\usepackage{graphicx}

\usepackage[all]{xy}
\usepackage{tikz}

\theoremstyle{plain}
\newtheorem{theorem}{Theorem}[section]

\newtheorem{proposition}[theorem]{Proposition}
\newtheorem{lemma}[theorem]{Lemma}
\newtheorem{corollary}[theorem]{Corollary}

\theoremstyle{definition}
\newtheorem{definition}[theorem]{Definition}

\newtheorem{example}[theorem]{Example}

\theoremstyle{remark}
\newtheorem{remark}[theorem]{Remark}

\newcommand{\rlift}{\bigr\rangle\bigr\rangle}
\newcommand{\llift}{\bigl\langle\bigl\langle}

\newcommand{\fL}{\mathfrak{L}}
\newcommand{\fP}{\mathfrak{P}}
\newcommand{\fT}{\mathfrak{T}}
\newcommand{\fI}{\mathfrak{I}}
\newcommand{\fD}{\mathfrak{D}}

\newcommand{\F}{\mathbb F}

\newcommand{\sheaf}[1]{\mathscr{#1}}
\newcommand{\DD}{\sheaf{D}}
\newcommand{\LL}{\sheaf{L}}
\newcommand{\OO}{\sheaf{O}}

\newcommand{\FF}{\sheaf{F}}

\newcommand{\NN}{\sheaf{N}}

\newcommand{\II}{\sheaf{I}}
\newcommand{\XX}{\sheaf{X}}
\newcommand{\sA}{\sheaf{A}}

\newcommand{\Z}{\mathbb Z}
\newcommand{\N}{\mathbb N}

\newcommand{\C}{\mathbb C}
\renewcommand{\P}{\mathbb P}
\newcommand{\Q}{\mathbb Q}

\DeclareMathOperator{\Def}{\mathrm{Def}}
\DeclareMathOperator{\Div}{\mathrm{Div}}

\DeclareMathOperator{\Pic}{\mathrm{Pic}}

\DeclareMathOperator{\Spec}{\mathrm{Spec}}

\newcommand{\Hom}{\mathrm{Hom}}

\begin{document}

\title[A rigid, not infinitesimally rigid surface with $K$ ample]{A rigid, not infinitesimally rigid surface with $K$ ample}

\author[B\"ohning]{Christian B\"ohning}
\address{Christian B\"ohning, Mathematics Institute, University of Warwick\\
Coventry CV4 7AL, England}
\email{C.Boehning@warwick.ac.uk}

\author[von Bothmer]{Hans-Christian Graf von Bothmer}
\address{Hans-Christian Graf von Bothmer, Fachbereich Mathematik der Universit\"at Hamburg\\
Bundesstra\ss e 55\\
20146 Hamburg, Germany}
\email{hans.christian.v.bothmer@uni-hamburg.de}

\author[Pignatelli]{Roberto Pignatelli}
\address{Roberto Pignatelli, Dipartimento di Matematica\\
Universit\`{a} di Trento\\
via Sommarive, 14\\
I-38123 Trento (TN)\\
Italy}
\email{Roberto.Pignatelli@unitn.it}

\date{\today}

% \subjclass[2010]{11E08, 11E20, 11E88, 14D06, 14F22, 14L35, 15A66, 16H05}

\begin{abstract}
We produce an example of a rigid, but not infinitesimally rigid smooth compact complex surface  with ample canonical bundle using results about arrangements of lines inspired by work of Hirzebruch, Kapovich \& Millson, Manetti and Vakil.
\end{abstract}

\maketitle

{\em  We feel honoured to dedicate this article to our teacher, colleague and friend Fabrizio Catanese on the occasion of his 70th birthday.}

%%%%%%%%%%%%%%%%%%%%%%%
\section{Introduction}\label{sIntroduction}
%%%%%%%%%%%%%%%%%%%%%%%

Let $X$ be a compact complex space. A (small) deformation of $X$ is a holomorphic proper flat map $\pi \colon \XX \to T$ where $T$ is a (germ of a) complex space with a marked point $o\in T$ such that $\pi^{-1}(o)\simeq X$. Recall that $\pi$ is called \emph{complete} if every other small deformation $\pi'\colon \XX' \to T'$, $o'\in T'$, $(\pi')^{-1}(o')\simeq X$, can be obtained from $\pi$ via lifting along a map $\varphi \colon T'\to T$, $\varphi (o')=o$, and that $\pi$ is called \emph{semiuniversal} if it is complete and the differential $d \varphi_{o'}$ is uniquely determined for the lifting map. As is well known, a semiuniversal $\pi \colon \XX \to T$ always exists and is unique up to isomorphism as was proven by Kuranishi \cite{Ku62} for $X$ a complex manifold and Grauert \cite{Grau74} in the general case. The base germ $(T, o)$ is called the \emph{Kuranishi space} of $X$ and sometimes denoted by $\mathrm{Def}(X)$. 

%If $X=S$ is a smooth minimal surface of general type, the local moduli space of $S$ is defined as the quotient $\mathrm{Def}(S)/\mathrm{Aut}(S)$ of $\mathrm{Def}(S)$ by the finite group of automorphisms $\mathrm{Aut}(S)$; if $S$ is the canonical model, then the local moduli space is the analytic germ of the Gieseker moduli space at the point corresponding to $S$. In general, the local moduli space only surjects onto the analytic germ of the Gieseker moduli space at $[S]$.

It was a long-standing problem of Morrow and Kodaira \cite[p. 45]{MoKo71} if there are compact complex manifolds $X$ with $\mathrm{Def}(X)$ a non-reduced point. Such $X$ are called rigid, but not infinitesimally rigid. The same question was independently recently posed in \cite[Question 1.5.B)]{BaCa18}; it was suggested there to look for a rigid minimal surface of general type with canonical system not ample. The first examples of such $X$ were given in \cite{BaPi18}; these $X$ are minimal surfaces of general type that arise as resolutions of certain nodal product-quotient surfaces. The nodal product-quotient surfaces are actually infinitesimally rigid, meaning that their Kuranishi space is a reduced point, but the Kuranishi spaces of the desingularisations acquire a non-reduced structure due to the presence of the $(-2)$-curves and their deformations. 

It is therefore natural to ask if there are also minimal surfaces of general type with $K$  ample such that the Kuranishi space is a non-reduced point. 
In this article we answer this in the affirmative. More precisely we prove the following 
\begin{theorem}
There exists a compact complex surface $\tilde{S}^\heartsuit$ with ample canonical bundle whose Kuranishi space is  isomorphic to $\mathrm{Spec}\, \C[x]/(x^2)$.

The surface $\tilde{S}^\heartsuit$  has invariants
\begin{align*}
K^2_{\tilde{S}^\heartsuit}&= 1,260,966&
\chi (\tilde{S}^\heartsuit) &= 151,802&
q (\tilde{S}^\heartsuit) &= 0
\end{align*}
\end{theorem}

Since $\Def(\tilde{S}^\heartsuit)$ is  isomorphic to $\mathrm{Spec}\, \C[x]/(x^2)$,  $\tilde{S}^\heartsuit$ gives a {\it minimal} solution to the problem of Morrow and Kodaira in the sense that $\Def(\tilde{S}^\heartsuit)$ is a non-reduced point of minimal length $2$ and minimal embedding dimension  $h^1 (\tilde{S}^\heartsuit,T_{\tilde{S}^\heartsuit})=1$. All the examples in \cite{BaPi18} had  $h^1 (\tilde{S}^\heartsuit,T_{\tilde{S}^\heartsuit})=6$ and in fact every example constructed with the same technique would have had $h^1 (\tilde{S}^\heartsuit,T_{\tilde{S}^\heartsuit})$ even by \cite[Proposition 2.8]{Polizzi}, \cite[Proposition 5.6]{ProdQuot}. We also note that the topological index $\tau  \left( \tilde{S}^\heartsuit \right)= K^2_{\tilde{S}^\heartsuit}-8\chi (\tilde{S})$ is positive, whereas the topogical index of all the examples in \cite{BaPi18}, as well as the topological index of any surface constructed with the same technique (see \cite[page158]{QEsurvey}) is negative.

We use a totally different method as the basis of our construction: we construct a line arrangement $\fL^\heartsuit$ in $\P^2$ whose associated incidence scheme is  isomorphic to $\mathrm{Spec}\, \C[x]/(x^2)$. We then associate a pair $(S, B)$ to this line arrangement, where $S$ is the blowup of $\P^2$ in the points belonging to at least three lines of the arrangement, and $B$ is a simple normal crossing divisor on $S$ consisting of the strict transforms of the lines in $\fL^\heartsuit$ and the exceptional divisors. We prove that the deformations of the pair $(S, B)$ are the same as the ones of $\fL^\heartsuit$ given by the incidence scheme. We then construct an abelian cover $\pi\colon \tilde{S}^\heartsuit \to S$ branched in $B$ with group $G=(\Z/7)^4$ by a method due to Pardini \cite{Par91}. Finally, we slightly refine methods introduced by Fantechi-Pardini \cite{FaPa97} and Manetti \cite{Man01} to prove that the Kuranishi space of $\tilde{S}^\heartsuit$ is the same as the one of $(S, B)$. Moreover, we show that the canonical bundle of $\tilde{S}^\heartsuit$ is ample and compute its invariants. 

It is not difficult to show that every rigid product-quotient surface is regular ($q=0$), so this happens for all examples in \cite{BaPi18}.
Our example is also regular, and our proof shows that this follows from a condition we imposed on the construction to ensure  that the Kuranishi space of $\tilde{S}^\heartsuit$ is the same as the one of $(S, B)$, and precisely condition a) in Theorem \ref{tDeformationsCoveringSurface}. However, irregular rigid surfaces that are abelian cover as above do exist, as shown in \cite{Hir83,BaCa19}, so this method could produce irregular rigid not infinitesimally rigid surfaces, provided that one suitably replaces that condition. We refer to \cite{Cat17} for a rather complete treatement of known results about the rigidity of compact complex manifolds, including a chapter devoted to the abelian covers of the plane branched on a line configuration.

\medskip

The roadmap of the paper is as follows. In Section \ref{sDefOfAbelianCovers} we recall the general theory of abelian covers and their deformations from Pardini \cite{Par91}, Fantechi-Pardini \cite{FaPa97}, Manetti \cite{Man01}. For our purposes we need a slightly improved version of \cite[Corollary 3.23]{Man01}, which we prove in Corollary \ref{cCriterion} by the same methods. 

In Section \ref{sIncidence} we consider specifically abelian covers constructed from line arrangements and their deformation theory. In this situation we translate the conditions of Corollary \ref{cCriterion} into a computationally accessible form given in Theorem \ref{tDeformationsCoveringSurface}. This is rather involved and one of the main technical ingredients of the proof of the main result.

In Section \ref{sIncidenceSchemes} we construct the line arrangement $\fL^\heartsuit$ and show that the associated incidence scheme is a non-reduced point isomorphic to $\mathrm{Spec}\, \C[x]/(x^2)$.
In Section \ref{sExampleSurface} we construct $\tilde{S}^\heartsuit$ and show that its Kuranishi space is again isomorphic to $\mathrm{Spec}\, \C[x]/(x^2)$. 
In Section \ref{sAmple} we prove that $\tilde{S}^\heartsuit$ has ample canonical divisor and compute its invariants in Theorem \ref{tKSquareChi}. 

 Finally, Appendix \ref{Appendix} contains the data needed to construct $\tilde{S}^\heartsuit$ explicitly.

%The construction of $\fL^\heartsuit$ is inspired by \cite{KaMi98} (going back maybe even to von Staudt's work on synthetic projective geometry and its connection to elementary arithmetic operations). However, we do not actually rely on any of the results in \cite{KaMi98}. 

The overall approach is inspired by \cite{Vak06}. However, Vakil's paper does not imply our result since he works with so-called \emph{singularity types}, which are equivalence classes of pointed schemes for the equivalence relation generated by elementary equivalences $(X, p) \sim (Y, q)$ if there is a smooth morphism $(X, p) \to (Y, q)$; thus, loosely speaking, he only works up to addition of smooth parameters, and then proves that every singularity type of finte type over ${\mathbb Z}$ can be found on some moduli space of surfaces of general type with $K$ ample. 

Moreover, Kapovich-Millson \cite{KaMi98} prove a version of Mn\"ev's universality theorem used by Vakil, where it is not necessary to add smooth parameters. However, the notion of incidence scheme used in \cite{KaMi98} (the \emph{space of finite based realisations of an abstract arrangement}) is not suited for our geometric construction since in \cite{KaMi98} neither points nor lines need to be distinct in the realisations they consider.

\medskip

\textbf{Acknowledgments.} We would like to thank Michael Kapovich for useful suggestions. 
The third author would like to thank I. Bauer and F. Catanese for pointing him to the question of Morrow and Kodaira and for several enlighting conversations related to it.

%%%%%%%%%%%%%%%%%%%%%%%%
\section{Abelian covers and their deformations}
\label{sDefOfAbelianCovers}
%%%%%%%%%%%%%%%%%%%%%%%%

Here we review results on abelian covers and their deformations, following Pardini \cite{Par91}, Fantechi-Pardini \cite{FaPa97} and Manetti \cite{Man01}.

The main goal of this section is Corollary \ref{cCriterion}, a variant of a  criterion of Manetti  giving, for a complex manifold $X$ with an action of a finite abelian group $G$,  necessary conditions for the small deformations of $X$  to correspond exactly to the deformations of the pair $(Y, D)$ where $Y=X/G$ is the quotient manifold and $D$ is the branch divisor suitably ``decorated''. 

We start by recalling  \cite[Definition 1.1]{Par91}. 
\begin{definition}
 Let $Y$ be a variety. An abelian cover of $Y$ with group $G$ is a finite map
$\pi \colon X \rightarrow Y$ together with a faithful action of $G$ on $X$  such that $\pi$ exhibits $Y$ as the quotient of $X$
via $G$.
\end{definition}
For every finite abelian group set $G^*=\Hom(G,\C^*)$ for its dual group. In the sequel we will always make the standing assumption that
\begin{center}
\fbox{$X$ is normal and $Y$ is smooth.}
\end{center}
Then $\pi$  is flat and the action of $G$ induces a splitting $\pi_*\OO_X=\bigoplus_{\chi \in G^*} L_\chi^{-1}$ for suitable line bundles $L_\chi \in \Pic(Y)$, where $G$ acts on $L_\chi^{-1}$ via the character $\chi$. The invariant summand $L_1$ is isomorphic to $\OO_Y$.

Let $R, D$ denote the ramification locus and the branch locus of $\pi$ respectively. $R$ consists of the points of $X$ that have nontrivial stabilizer.
\cite{Par91} associates to each irreducible component of $R$ and $D$ a pair $(H,\psi)$ where $H$ is a cyclic subgroup and $\psi$ is a generator of $H^*$ as follows.
\begin{definition}
 Let $T$ be an irreducible component of $R$. Then the inertia group $H$ of $T$ is defined as $H=\left\{ h \in G | hx=x\  \forall x \in T  \right\}$.
 \end{definition}
 \begin{lemma}  Let $T$ be an irreducible component of $R$. 
 Then the inertia group $H$ of $T$ is cyclic and there is a unique character $\psi$ generating $H^*$ such that there is a parameter $t$ for $\OO_{X,T}$ satisfying 
 $ ht=\psi(h)t$ for all $h  \in H$.
 
 In other words, $\psi$ is the character by which $H$ acts on the cotangent space $\mathfrak{m}_T/\mathfrak{m}_T^2$, $\mathfrak{m}_T\subset \OO_{X, T}$ the maximal ideal in the local ring of $T$ in $X$. 
 \end{lemma}
 \begin{proof}\cite[Lemmata 1.1 and 1.2]{Par91}
 \end{proof}

Since the group $G$ is commutative, if $E$ is a component of $D$, then all the components of $\pi^{-1}(E)$ 
have the same inertia group $H$ and character $\psi \in H^*$. This splits the branch locus as a sum of (reduced effective but still possibly reducible) divisors
$D=\sum_{H,\psi} D_{H,\psi}$. 

Following \cite[Section 2]{FaPa97} we set, for all $m\in \N$, $\zeta_m:=e^\frac{2\pi i}{m}$ and we observe that there is a bijection among $G$ and the set of the pairs $(H,\psi)$ where $H$ is a cyclic subgroup of $G$ and $\psi$ is a generator of $H^*$, the bijection being given by $(H,\psi) \mapsto g$  where $g \in H$ is the generator of $H$ such that $\psi(g)=\zeta_{\#H}$. This allows to set $D_g:=D_{H,\psi}$ and write 
\[
D=\sum_{g \in G} D_g
\]

We note $D_0=0$. We find this notation more convenient than the one with the $D_{H,\psi}$ so will formulate all results in this notation.

Following again  \cite[Section 2]{FaPa97}, for $\chi \in G^*$, $g \in G$, let $a^g_\chi$ be the unique integer $0\leq a_\chi^g \leq o(g)-1$ such that 
$\chi(g)=\zeta_{o(g)}^{a^g_\chi}$. Here $o(g)$ denotes the order of $g$. Define $\epsilon_{\chi,\chi'}^g:=\left\lfloor \frac{a^g_\chi+a^g_{\chi'}}{o(g)}\right\rfloor \in \left\{0,1\right\}$.

The following is Pardini's structure theorem for abelian covers.

\begin{theorem}\label{tStructureTheorem}
Let $G$ be an abelian group. 

Let $Y$ be a smooth variety, $X$ a normal one, and let $\pi \colon X\rightarrow Y$ be 
an abelian cover with group $G$. We have associated to $\pi$ two functions
\begin{gather}\label{fCoverData}
\DD \colon G \to \Div_+(Y) , \quad \LL \colon G^{*} \to \Pic (Y)
\end{gather}
where $\mathrm{Div_+}(Y)$ is the subset of the group $\Div(Y)$ formed by the  effective divisors. 
Setting $\DD(g)=D_g$ and $\LL(\chi)=L_\chi$ we have $D_0=L_1=0$ and $D=\sum D_g$ reduced.

Then the following set of linear equivalences is satisfied:
\begin{equation}\label{fPardiniEqn}
L_\chi+L_{\chi'} =L_{\chi \chi'}+\sum_g \epsilon_{\chi,\chi'}^g D_g.
\end{equation}
Conversely, for every smooth $Y$, to any  two functions $\DD \colon G \to \Div_+(Y)$, $\LL \colon G^{*} \to \Pic (Y)$ satisfying (\ref{fPardiniEqn}) 
with $D_0=L_1=0$, $D:=\sum D_g$  reduced, there is an abelian cover $\pi \colon X \rightarrow Y$ with  $X$  normal
with associated functions $\DD$ and $\LL$. 

Moreover,  if $Y$ is complete, then $\pi \colon X \rightarrow Y$ is unique  up to isomorphism of Galois covers.
\end{theorem}
\begin{proof}
\cite[Theorem 2.1 and Corollary 3.1]{Par91}
\end{proof}

The smoothness of $X$ is easy to translate in this setting.

\begin{proposition}\label{pSmooth}
In the situation of Theorem \ref{tStructureTheorem}  $X$ is smooth if and only if
\begin{enumerate}
\item the divisor $D=\sum D_g$ is a smooth normal crossing divisor;
\item if $D_{g_1} \cap D_{g_2} \cap \cdots \cap D_{g_r} \neq \emptyset$ then the map $\langle g_1 \rangle \oplus \langle g_2 \rangle \oplus \cdots  \oplus \langle g_r \rangle \rightarrow G$ is injective.
\end{enumerate}
\end{proposition}
 \begin{proof}
 \cite[Prop. 3.1]{Par91}
 \end{proof}

 Now we want to study infinitesimal
deformations of $\pi \colon X \rightarrow Y$ obtained by ``moving" the branch divisors, following \cite{Man01}. For the rest of this Section, 
we assume for simplicity from now on 
\begin{center}
\fbox{$X$ and $Y$ are smooth and $\pi$ is totally ramified}
\end{center}
in the following sense:
\begin{definition}
 A cover is said to be totally ramified if the inertia groups of all the components of $R$ generate $G$.
\end{definition}
We introduce the following notation.
\begin{definition}
We define $S_\chi:=\left\{ g \in G | \chi(g) \neq \zeta_{o(g)}^{-1}\right\}$.
\end{definition}
Then $G$ acts on $\pi_*T_X$, where $T_X$  denotes  the sheaves of the holomorphic vector fields on $X$, splitting it as follows.

\begin{lemma}\label{lSplitsTX}
\begin{align*}
(\pi_* T_X)^{inv} &\cong T_Y(-\log D)&
(\pi_* T_X)^{\chi} &\cong T_Y\left(-\log \sum_{g \in S_\chi} D_g\right) \otimes L_\chi^{-1}
\end{align*}
\end{lemma}
\begin{proof}
\cite[Proposition 4.1]{Par91}.
\end{proof}
It follows among other things
\begin{lemma}\label{lH0TX=0}
If for all $\chi \in G^*$, $H^0(Y, T_Y \otimes L_\chi^{-1})=0$, then $H^0(X,T_X)=0$.  
\end{lemma}

 Assume $H^0(X,T_X)=H^0(Y,T_Y)=0$.
Let $Art$  be the category of local Artinian $\C-$algebras and denote by
$Def_X$, $Def_Y \colon Art \rightarrow Sets$ the functors of deformations  of $X$, $Y$ respectively.
Under the assumption $H^0(X,T_X)=H^0(Y,T_Y)=0$ these are prorepresentable, and prorepresented by the Kuranishi families $\Def_X$, $\Def_Y$ of $X$, $Y$ respectively. Moreover for $i=1,2$
\begin{align}\label{fT^iDef}
T^i Def_X&=H^i(X,T_X)&
T^i Def_Y&=H^i(Y,T_Y)
\end{align}
where as usual we denote by $T^1$ the tangent space and by $T^2$ the obstruction
space arising from the cotangent complex. Please note that $G$ acts on $T^i \Def_X$ and we can write the corresponding eigenspaces as cohomology groups on $Y$ by Lemma \ref{lSplitsTX}.

Let $Def_{(Y,\DD)} \colon Art \rightarrow Sets$ be the functor of deformations of the closed
inclusions $D_g \rightarrow Y$; more precisely for $A$ in $Art$, $Def_{(Y,\DD)}(A)$ is the set of 
isomorphism classes of:
\begin{enumerate}
\item a deformation of $Y$, $Y_A \rightarrow \Spec A$
\item for every $g \in G$ a closed embedding $D_{A,g} \subset Y_A$ extending $D_g$
\end{enumerate}

Note that  $Def_{(Y,\DD)}$ is prorepresented by the fibre product $\Def_{(Y,\DD)}$  of the corresponding
relative Hilbert schemes of the Kuranishi family of $Y$.

\begin{lemma} 
For $i=1,2$, $D=\sum D_g$
\[T^i Def_{(Y,\DD)} \cong H^i(Y,T_Y(-\log  D)).
\]
\end{lemma}
\begin{proof}
This is \cite[(3.5.1)]{FaPa97}, where $Def_{(Y,\DD)}$ is named $Dgal_X$. 
\end{proof}

Comparing it with Lemma \ref{lSplitsTX} and \ref{fT^iDef} we see that $T^i Def_{(Y,\DD)}$ is isomorphic to the invariant part of $T^i \Def_X$. Indeed
this corresponds to the natural relation among the functors $Def_{(Y,\DD)}$ and $Def_X$  associating to every infinitesimal deformation of $(Y,\DD)$ the infinitesimal deformation of $X$ obtained by ``moving'' $Y$ and the branching divisors in it. This has been formalized as follows

\begin{lemma}\label{lInjonT1} Assume $H^0(X,T_X)=H^0(Y,T_Y)=0$. 

Then there is a natural transformation of functors $\eta \colon Def_{(Y,\DD)}\rightarrow Def_X$ acting on the $T^j$s by mapping isomorphically 
$T^i Def_{(Y,\DD)}$ in the invariant part of $T^i Def_X$.
\end{lemma}
\begin{proof}
This is \cite[Theorem 3.22]{Man01}. 

More precisely,  M. Manetti states  the result  assuming $G$ of the form $\left( \Z/2\Z \right)^r$. However his proof works for every abelian group. 
\end{proof}

 It follows the following
 \begin{proposition}\label{pCriterion}
Let $\pi \colon X \rightarrow Y$ be a totally ramified abelian cover, with $X$ and $Y$ smooth and  $H^0(X,T_X)=H^0(Y,T_Y)=0$.
If for all $\chi \in G^* \setminus 1$ 
\[
H^1\left(Y,T_Y\left(-\log \sum_{g \in S_\chi} D_g\right) \otimes L_\chi^{-1}\right)=0
\]
then $\Def_{(Y,\DD)} \cong \Def_X$.
 \end{proposition}
\begin{proof}
This is essentially explained in \cite[page 58]{Man01}, we sketch the argument here for the convenience of the reader.

The natural maps $T^i Def_{(Y,\DD)} \rightarrow T^i Def_X$ are injective by Lemma \ref{lInjonT1} and the assumed cohomological vanishing ensures by Lemma \ref{lSplitsTX} the surjectivity on $T^1$s.  By the standard smoothness
criterion $\eta$ is smooth. 
Since $H^0(X,T_X)=H^0(Y,T_Y)=0$ then $Def_X$ is prorepresentable. Since $\eta$ induces an isomorphism on $T^1$s
the map $\Def_{(Y,\DD)} \rightarrow \Def_X$ is an isomorphism.
\end{proof}
We deduce the following useful criterion, a slight modification of \cite[Corollary 3.23]{Man01}.
\begin{corollary}\label{cCriterion}
Let $\pi \colon X \rightarrow Y$ be a totally ramified abelian cover, with $X$ and $Y$ smooth. Assume
\begin{enumerate}
\item for all $\chi \in G^*$, $H^0(Y,T_Y \otimes L_\chi^{-1})=0$;
\item for all $\chi \in G^*$, $\chi \neq 1$, for all $g \in S_\chi$, $H^0(D_g,\OO_{D_g}(D_g) \otimes L_\chi^{-1})=0$;
\item  for all $\chi \in G^*$, $\chi \neq 1$, the map
\[
H^1\left( T_Y \otimes L_\chi^{-1} \rightarrow \left( \bigoplus_{g \in S_\chi} \OO_{D_g}(D_g)\right) \otimes L_\chi^{-1} \right)
\]
induced by the natural maps among the tangent bundle of $Y$ and the normal bundles  of the curves $D_g$, is  injective.
\end{enumerate}
Then $\Def_{(Y,\DD)} \cong \Def_X$.
\end{corollary}
\begin{proof}
We have assumed the vanishing of $H^0(Y,T_Y)$ (condition $a)$ for $\chi=1$); the vanishing of $H^0(X,T_X)$ follows by Lemma \ref{lH0TX=0}. 

We only need then to prove the cohomological vanishing assumed in the statement of Proposition \ref{pCriterion}.
This follows by considering the cohomology exact sequence associated to the short exact sequence of $\OO_Y$-modules
\[
0
\rightarrow
T_Y\left(-\log \sum_{g \in S_\chi} D_g\right) 
\rightarrow
T_Y
\rightarrow
\bigoplus_{g \in S_\chi} \OO_{D_g}(D_g)
\rightarrow
0
\]
twisted by $L_{\chi}^{-1}$.
\end{proof}

 The main difference with  \cite[Corollary 3.23]{Man01} is that our condition $c)$ is weaker than Manetti's condition $H^1(T,T_Y \otimes L_\chi^{-1})=0$.
 Manetti's condition is easier to check but it fails in the examples in the next sections. 
 
 A second difference is that Manetti's criterion is stated for $G$ of the form $(\Z/2\Z)^r$: \cite{Vak06} had already noticed that his results could be easily extended to any abelian group. Finally, condition $b)$ is replaced in Manetti's criterion by some conditions on $Y$ that are used in his proof to show exactly the vanishing $b)$. Since in our examples the divisors $D_g$ are smooth curves with all components rational, it is easier to check directly $b)$. 
 
%%%%%%%%%%%%%%%%%%%%%%%%%%%%%%%%
\section{Properties of abelian covers from line arrangements}
\label{sIncidence}
%%%%%%%%%%%%%%%%%%%%%%%%%%%%%%%%

Now we specialize the results of the previous Section to the case of interest in the sequel.

From now on we are going to consider only elementary $p-$groups: $G:=(\Z/p)^r$. For this class of abelian groups we find convenient to switch to the additive notation. To mark this difference with the previous section, we are going to denote the dual of $G$ by $G^\vee$.
\begin{definition}
Denote by 
\begin{align*}
\langle \cdot , \cdot \rangle \colon  G^{\vee} \times G & \to \Z/p, \\
  (\chi , g) & \mapsto \langle \chi, g \rangle = \chi (g)
\end{align*}
the natural pairing between characters and group elements. We write 
\begin{align*}
\langle\langle \cdot , \cdot \rangle\rangle \colon  G^{\vee} \times G & \to \{ 0, 1, \dots , p-1\} \subset \Z, \\
  (\chi , g) & \mapsto \langle\langle \chi, g \rangle \rangle 
\end{align*}
for the natural lift of the preceding pairing to $\Z$ taking values in the set $\{ 0 , 1, \dots , p-1\}$. 
\end{definition}

The abelian covers will be constructed from certain line arrangements in $\P^2$.  

\begin{definition}\label{dLineArrangementETC}
\begin{enumerate}
\item
An \emph{arrangement of lines} $\mathfrak{L}$ in $\P^2$ is simply a finite set of lines in $\P^2$.
\item
Given a line arrangement $\mathfrak{L}$, we call the points in $\P^2$ that lie on three or more of the lines in $\mathfrak{L}$ the \emph{singular points} of the arrangement $\mathfrak{L}$. 
\end{enumerate}
\end{definition}

\begin{definition}\label{dGeneralisedIncidence}
Suppose we are given $m$ points $p_1, \dots , p_m$ and $n$ lines $L_1, \dots , L_n$ in $\P^2$ which we think of as being variable. Moreover, assume that we are also given further $m'$ points $q_1, \dots , q_{m'}$ and further $n'$ lines $M_1, \dots , M_{n'}$ which we think of as being fixed beforehand. 

We then define a \emph{generalised incidence scheme} $\mathfrak{I}$ \emph{with associated fixed data $q_1, \dots , q_{m'}$ and $M_1, \dots , M_{n'}$} as a closed subscheme of 
$$(\P^2)^m\times ((\P^2)^{*})^n$$
with coordinates $(p_1, \dots , p_m; L_1, \dots , L_n)$
defined by equations expressing the fact that a point $r\in \{ p_1, \dots , p_m, q_1, \dots , q_{m'}\}$ lies on a line $N\in \{ L_1, \dots , L_n, M_1, \dots , M_{n'}\}$. 
Hence these equations are of bidegree $(1,1)$, $(1,0)$, $(0,1)$, or $(0,0)$ in the coordinates on $(\P^2)^m\times ((\P^2)^{*})^n$. 
\end{definition}

\begin{definition}\label{dIncScheme}
An \emph{incidence scheme} is a generalised incidence scheme associated to fixed data the points $q_1=(1:0:0), q_2=(0:1:0), q_3=(0:0:1), q_4=(1:1:1)$.
\end{definition}

\begin{definition}\label{dLineArrIncScheme}
Given a line arrangement $\mathfrak{L}$ in $\P^2$ with $q_1=(1:0:0), q_2=(0:1:0), q_3=(0:0:1), q_4=(1:1:1)$ among the intersections of lines, we define the associated incidence scheme $\mathfrak{I}(\mathfrak{L})$ by associating to each line a variable line and to each singular point different from $q_1, \dots , q_4$ of the line arrangement a variable point, subject to the incidences given by the line arrangement. 
\end{definition}

\begin{definition}\label{dCoversFromLineArrangements}
Let $\mathfrak{L}=\{ L_1, \dots , L_n\}$ be a line arrangement in $\P^2$, $p_1, \dots , p_m$ the singular points of $\mathfrak{L}$ and $H$ the class of a line in $\Pic (\P^2)$. Let $G= (\Z /p)^r$ be an elementary abelian $p$-group. We denote by $G^\vee$ its dual. 
\begin{enumerate}
\item
We denote by $\sigma\colon S\to \P^2$ the blow up of $\P^2$ in $p_1, \dots , p_m$. 
\item
We let $B$ be the union of the strict transforms $\bar{L}_i$ of the $L_i$, $i=1, \dots , n$ and the exceptional divisors $E_{\nu}$ over $p_{\nu}$, $\nu=1, \dots , m$. Let $\mathfrak{D}$ be the set 
\[
\mathfrak{D} = \{ E_1, \dots , E_m\} \cup \{ \bar{L}_1, \dots , \bar{L}_n\}.
\] 
\end{enumerate}
\end{definition}

Our aim is to construct a smooth abelian $G$-cover $\pi\colon \tilde{S}\to S$ with branch locus $B$. We now introduce a compact way to produce building data in our special case. 

\begin{definition}\label{dDivisibility}
Let $S$ and $\fD$ be as in Definition \ref{dCoversFromLineArrangements} and let 
\[
	\lambda \colon \fD \to G \cong  (\Z/p)^r 
\]
be a map (the letter $\lambda$ is chosen to suggest ``label"). We say that {\sl $\lambda$ satisfies the divisibility condition} if for every $\chi \in G^\vee$ 
\[
	\sum_{D \in \fD} \llift \chi, \lambda(D) \rlift D
\]
is divisible by $p$ in $\Pic(S)$. In this case we define for every $\chi \in G^\vee$ the line bundle:
\[
	L_\chi = \frac{1}{p}  \sum_{D \in \fD} \llift \chi, \lambda(D) \rlift \OO_S(D)
\]
on $S$. This uniquely determines $L_\chi$ since $\Pic (S)$ has no torsion.
\end{definition}

\begin{lemma}\label{lLambdaUnique}
For every set of values $g_1, \dots, g_{n-1} \in G\backslash \{0\}$ there exist a unique 
map
\[
	\lambda \colon \fD \to G \cong  (\Z/p)^r 
\]
with $\lambda(\bar{L}_{i}) = g_i$ for $i \in \{1,\dots,n-1\}$, satisfying the divisibility condition. 
\end{lemma}

\begin{proof} 
Fix $\chi \in G^{\vee}$ and let $\lambda$ be a map satisfying the divisibility condition. 

Recall that $\Pic(S)$ is a free $\Z-$module of rank $m+1$, generated by the classes of the exceptional divisors $E_\nu$ and by $\sigma^*H$.
To determine the divisibility of an element of $\Pic(S)$ we calculate its coefficients with respect to this basis  and we impose their divisibility.

Since $\sigma^* H$ occurs in the class of $\bar{L}_i$ with coefficient $1$ and in the class of $E_\nu$ with coefficient $0$, the coefficient of $\sigma^* H$ in $L_\chi$ is
\[
	\frac{1}{p} \sum_{i=1}^{n}  \llift \chi,\lambda(\bar{L}_i) \rlift .
\]
In particular 
\[
	\sum_{i=1}^{n}  \llift \chi,\lambda(\bar{L}_i) \rlift 
\]
is divisible by $p$. Therefore in $\Z/p$ 
\[
	0= \sum_{i=1}^{n}  \langle \chi,\lambda(\bar{L}_i) \rangle = \langle \chi , \sum_{i=1}^n \lambda(\bar{L}_i) \rangle .
\]
Since this holds for all $\chi$
\[
\sum_{i=1}^n \lambda(\bar{L}_i) =0  \iff \lambda (\bar{L}_n)= -\sum_{i=1}^{n-1} \lambda(\bar{L}_i) .
\]
Similarly, the coefficient of $E_{\nu}$ in $\bar{L}_i$ is $-1$ if $p_{\nu}\in L_i$ and $0$ otherwise. Therefore the coefficient of $E_{\nu}$ in $L_{\chi}$ is 
\[
\frac{1}{p} \left( -\sum_{i \mid p_{\nu}\in L_i} \llift \chi,\lambda(\bar{L}_i) \rlift + \llift \chi,\lambda(E_{\nu}) \rlift \right) .
\]
As above we obtain
\[
\lambda(E_{\nu})  =  \sum_{i \mid p_{\nu}\in L_i} \lambda(\bar{L}_i) .
\]
Hence, all values of $\lambda$ are determined by the $g_i$, and conversely prescribing values $\lambda(\bar{L}_{i}) = g_i$ for $i \in \{1,\dots,n-1\}$ arbitrarily, we obtain an $\lambda$ satisfying the divisibility condition using the preceding assignments. 
\end{proof}

\begin{lemma}\label{lCoeffExceptional}
Let $\chi \in G^\vee$. Then the coefficient of $\sigma^* H$ in $L_\chi$ is
\[
	\left\lceil \frac{1}{p} \sum_{i=1}^{n-1}  \llift \chi,\lambda(\bar{L}_i) \rlift \right\rceil
\]
similarly the coefficient of $E_\nu$ in $L_\chi$ is
\[
	 -\left\lfloor \frac{1}{p} \sum_{i| p_{\nu}\in L_i}  \llift \chi,\lambda(\bar{L}_i) \rlift \right\rfloor .
\]
In particular, the coefficient of each $E_{\nu}$ in every $L_{\chi}$ is nonpositive.
\end{lemma}

\begin{proof}
By definition of $L_\chi$ the coefficient of $\sigma^* H$ in $L_\chi$ is
\[
	\frac{1}{p} \sum_{i=1}^{n}  \llift \chi,\lambda(\bar{L}_i) \rlift .
\]
In particular 
\[
	\frac{1}{p} \sum_{i=1}^{n}  \llift \chi,\lambda(\bar{L}_i) \rlift  = \frac{1}{p}  \sum_{i=1}^{n-1}  \llift \chi,\lambda(\bar{L}_i) \rlift  + \frac{\llift \chi,\lambda(\bar{L}_n) \rlift}{p}
\]
 and
\[
	0 \le \frac{\llift \chi,\lambda(\bar{L}_n) \rlift }{p} < 1.
\]
The formula of the lemma follows. Similarly, the coefficient of $E_{\nu}$ is 
\[
	   - \frac{1}{p} \sum_{i| p_{\nu}\in L_i}  \llift \chi,\lambda(\bar{L}_i) \rlift  + \frac{1}{p} \llift \chi,\lambda(E_\nu) \rlift 
\]
The second formula of the lemma follows. 
\end{proof}

\begin{definition}\label{dLambda}
In the situation of Definition \ref{dCoversFromLineArrangements} let 
\[
	\lambda \colon \fD \to G\backslash \{0\} \cong (\Z/p)^r \backslash \{0\}
\]
be a map (note that we now exclude $0\in G$ as a permissible value for $\lambda$). We say, that {\sl $\lambda$ satisfies the injectivity condition} if
the values of $\lambda$ define distinct projective points in $\P^{r-1} (\F_p)$.
Also we say that {\sl $\lambda$ satisfies the spanning condition} if the image
of $\lambda$ spans $G$.
\end{definition}

If a $\lambda$ (with values in $G\backslash \{0\}$) satisfies the injectivity and divisibility conditions, one obtains building data for an abelian $G$-cover in the sense of Section \ref{sDefOfAbelianCovers} as follows: we get maps
\[
\DD \colon G \to \mathrm{Div}^+ (S), \: \LL \colon G^{\vee}\to \mathrm{Pic}(S)
\]
by putting $\DD (g)=D_g = D$ if $\lambda (D)= g$ and $D_g=0$ otherwise. Moreover, $\LL (\chi ) =L_{\chi}$ is then as in Definition \ref{dDivisibility}. 

\begin{theorem}\label{tStructureTheoremSpecial}
Let
\[
	\lambda \colon \fD \to G\backslash \{0\} \cong (\Z/p)^r \backslash \{0\}
\]
be a map satisfying the divisibility, injectivity and spanning conditions.
Then there exists a finite flat totally ramified Galois cover $\pi \colon \tilde{S}_\lambda \to S$ with group $G$, branch locus $\sum_{D \in \fD} D$, and with the covering surface $\tilde{S}_\lambda$ smooth, with the property that 
\begin{align*}
& \pi_* \OO_{\tilde{S}_\lambda} = \bigoplus_{\chi \in G^{\vee}} L_{\chi}^{-1}.
\end{align*}
\end{theorem}

\begin{proof}
This is a direct consequence of Theorem \ref{tStructureTheorem} and Proposition \ref{pSmooth}; compare also \cite[Prop. 2.1]{Par91}, \cite[Prop. 3.1]{Par91} and \cite[Prop. 4.1]{Vak06} for more details. 
\end{proof}

From now on we assume for the rest of the section that we are given a cover $\pi\colon \tilde{S}_{\lambda} \to S$ constructed by the method in Theorem \ref{tStructureTheoremSpecial}, and to simplify notation we write $\tilde{S}=\tilde{S}_{\lambda}$. 

\begin{proposition}\label{pIsoDefs}
Keeping the previous assumptions of this section, the Kuranishi space $\mathrm{Def}_{(S, \fD )}$ of deformations of $S$ together with the closed embeddings $D \hookrightarrow S$, $D\in \fD$, is locally analytically isomorphic to the germ of the incidence scheme $\mathfrak{I}(\mathfrak{L})$ around the distinguished point $\omega_{\mathfrak{L}}$ determined by $\mathfrak{L}$.
\end{proposition}

\begin{proof}
We follow rather closely the proof in \cite[Prop. 3.2]{Vak06}. There is a natural morphism of germs
\[
\psi\colon (\mathfrak{I}, \omega_{\mathfrak{L}}) \to \mathrm{Def}_{(S, \fD )}
\]
and the task is to construct a local inverse to this.

Given a deformation of $S$ over a germ $\Delta$ together with compatible deformations of the closed embeddings $D \hookrightarrow S$, the $(-1)$-curves in $S$ remain $(-1)$-curves under the deformation and can be blown down in the family as shown in \cite[Prop. 3.2]{Vak06}. We can identify the blown down family with $\Delta\times \P^2$ by choosing the images of the deformations of $E_1, \dots , E_4$ as a projective basis. The images of the deformations of $E_5, \dots , E_m$ give deformations of $p_5, \dots , p_m$. 
The images of the deformations of the $\bar{L}_i$ give deformations of the given lines $L_i$ in $\P^2$. Since the intersection numbers $E_{\nu}. \bar{L}_i$ remain constant, all incidences $\mathfrak{I}(\mathfrak{L})$ are satisfied for the deformations of $p_{\nu}$, $L_i$. No new incidences occur since we consider a small deformation. 

Both the incidence scheme and the Kuranishi space above are universal objects representing two deformation functors. The above reasoning shows that these two functors are isomorphic, hence the representing germs of analytic spaces are isomorphic, together with their potentially nonreduced structure.
\end{proof}

Our ultimate goal now in the rest of the Section is to produce a computationally checkable criterion that implies the conditions of Corollary \ref{cCriterion}. For this, we first need to compute some cohomology groups.

\begin{lemma}\label{lDirectImages}
Let $\sigma\colon S\to \P^2$ be the blow up of $\P^2$ in points $p_1, \dots , p_m$ with corresponding exceptional divisors $E_1, \dots , E_m$. Let
\[
\LL = \OO \left(- \sum_{\nu=1}^m h_{\nu} E_{\nu}\right)
\]
with $h_{\nu}\ge -1$ for all $\nu$. 
Then 
\[
R\sigma_* \LL = \bigcap_{\nu\mid h_{\nu}\ge 1} \II_{p_{\nu}}^{h_{\nu}} 
\]
where the equality is in the derived category $D^b(\P^2)$, $\II_{p_{\nu}}$ is the ideal sheaf of the point $p_{\nu}$, and the intersection is taken inside $\OO_{\P^2}$.
\end{lemma}

\begin{proof}
Since all fibres of $\sigma$ have dimension $\le 1$, $R^k \sigma_* \LL =0$ for $k\ge 2$ anyway. It suffices to check the remaining assertions for the case where there is just one exceptional divisor $E$ mapping to one point $p\in \P^2$ by working stalk-wise and gluing. 

\medskip

We have the exact sequence
\[
0\to \OO_S \to \OO_S (E) \to \OO_E (-1)\to 0.
\]
Since $h^0 (\OO_E(-1))=h^1(\OO_E(-1))=0$, we have $R\sigma_* \OO_E(-1)=0$. Therefore
\[
R\sigma_* \OO_S(E) = R\sigma_* \OO_S =\OO_{\P^2}.
\]

The fact that $\sigma_*  \OO (-hE)=\II_p^h$ for $h\ge 0$ is a local calculation. 
Let us now prove $R^1 \sigma_* \OO (-hE)= 0$ for $h\ge 0$.

We prove this by induction on $h$, the case $h=0$ being clear. The exact sequence
\[
0\to \OO_S (-(h+1)E) \to \OO_S (-hE) \to \OO_E (h)\to 0
\]
gives
\[
0\to \II_p^{h+1} \to \II_p^h \xrightarrow{\alpha} \OO_p^{h+1} \to R^1\sigma_* \OO_S (-(h+1)E) \to 0
\]
The map $\alpha$ is surjective by a dimension count. Hence $R^1\sigma_* \OO_S (-(h+1)E)=0$. 
\end{proof}

\begin{definition}\label{dIdealChi}
In the situation of Theorem \ref{tStructureTheoremSpecial}, write 
\[
L_{\chi}\otimes K_S = \OO_S \left( d^{\chi}\sigma^*H - \sum_{\nu=1}^m h_{\nu}^{\chi} E_{\nu}\right).
\]
Let
\[
\II_{\chi} := \bigcap_{\nu \mid h_{\nu}^{\chi}\ge 1} \II_{p_{\nu}}^{h_{\nu}^{\chi}}. 
\]
\end{definition}

\begin{corollary}\label{cLChi}
For all $i\ge 0$ we have
\[
H^i (S, L_{\chi}\otimes K_S\otimes \sigma^* (aH)) = H^i (\P^2, \II_{\chi}(d^{\chi}+a)).
\]
\end{corollary}

\begin{proof}
By Lemma \ref{lCoeffExceptional} each $h^{\chi}_{\nu}\ge -1$ since $E_{\nu}$ occurs with nonpositive coefficient in $L_{\chi}$ and coefficient $1$ in $K_S$. The result then follows from Lemma \ref{lDirectImages} and the derived projection formula
\[
H^i (S, \sigma^* \LL \otimes \FF ) \simeq H^i (\P^2, \LL\otimes R\sigma_* \FF )
\]
for $\LL$ locally free and $\FF$ any coherent sheaf. 
\end{proof}

For the following compare \cite[Lecture 14]{Mum66}. 

\begin{definition}\label{dCastelnuovoMumfordRegularity}
A coherent sheaf $\FF$ on $\P^n$ is called $r$-regular if
\[
H^{i}(\P^{n},\FF(r-i))=0
\]
whenever $i>0$. The Castelnuovo-Mumford regularity of $\FF$, denoted by $\mathrm{reg} (\FF )$, is the smallest integer $r$ such that $\FF$ is $r$-regular. 
\end{definition}

\begin{theorem}\label{tCastelnuovo}
If a coherent sheaf is $r$-regular then $\FF (r)$ is generated by its global sections and the natural map 
\[
H^0 (\FF (k-1)) \otimes H^0 (\OO (1)) \to H^0 (\FF (k))
\]
is surjective for $k > r$.
\end{theorem}
\begin{proof}
 This is \cite[Prop. on page 99, parts a) and a')]{Mum66}.
\end{proof}

We now define, for a given $\chi\neq 0$ two maps $\alpha_{\chi}$ and $\beta_{\chi}$ that will later be used in the proof of Theorem \ref{tDeformationsCoveringSurface}. 

We first consider the following diagram for each $D\in \fD$ that is not an exceptional divisor:
{\small 
\begin{gather}\label{dFund1}
\xymatrix{
& 0\ar[d] & 0\ar[d] & 0\ar[d] \\
0\ar[r] &T_S (-D) \otimes L_{\chi}^{-1} \ar[r]\ar[d] & \sigma^* T_{\P^2}(-D)\otimes L_{\chi}^{-1} \ar[r]\ar[d] & \OO(-D)\otimes \bigoplus_{\nu} \OO_{E_{\nu}}(1) \otimes  L_{\chi}^{-1} \ar[r]\ar[d]^{\psi_D} & 0\\
0\ar[r] & T_S \otimes L_{\chi}^{-1} \ar[r]\ar[d]^{\zeta_D} & \sigma^* T_{\P^2} \otimes L_{\chi}^{-1} \ar[r]\ar[d] & \bigoplus_{\nu} \OO_{E_{\nu}}(1) \otimes L_{\chi}^{-1}  \ar[r]\ar[d]^{\eta_D} & 0\\
0\ar[r] & (T_S  \otimes L_{\chi}^{-1})\mid_{D} \ar[r]^-{\varphi_D}\ar[d] & (\sigma^* T_{\P^2} \otimes L_{\chi}^{-1})\mid_{D} \ar[r]\ar[d] & \bigoplus_{\nu} \OO_{E_{\nu}\cap D} (1)\otimes L_{\chi}^{-1} \ar[r]\ar[d] & 0\\
 & 0 & 0 & 0
}
\end{gather}}
The upper and middle rows are 
\begin{gather}\label{fTangentSequence}
0\to T_S \to \sigma^* T_{\P^2} \to \bigoplus_{\nu} \OO_{E_{\nu}}(1) \to 0
\end{gather}
tensored with appropriate line bundles, and the morphism between those rows is induced by the ideal sheaf sequence
\[
0\to \OO_S (-D) \to \OO_S \to \OO_{D}\to 0. 
\]
The left and middle columns are exact because the result from tensoring this ideal sheaf sequence by vector bundles. Tensoring the middle row with $\OO_{D}$ gives the third row; note that exactness on the left follows from the vanishing of $Tor_1 (\OO_{E_{\nu}}, \OO_{D})$ because $D$ and $E_{\nu}$ intersect transversely. The snake lemma shows that the kernel and cokernel of $\varphi_D$ are the same as the kernel and cokernel of $\psi_D$, hence completes the diagram.

We get the commutative diagram
\begin{gather*}
\xymatrix{
\bigoplus_{\nu} H^0 (\OO_{E_{\nu}}(1) \otimes L_{\chi}^{-1} ) \ar[r]^{\quad\quad \delta_{\chi}} \ar[d]^{\beta_D} & H^1 (T_S\otimes L_{\chi}^{-1})\ar[d]^{\gamma_D} \\
\bigoplus_{\nu} H^0 ( \OO_{E_{\nu}\cap D}(1)\otimes L_{\chi}^{-1} ) \ar[r]\ar[rd]_{\alpha_D} &  H^1 ( (T_S  \otimes L_{\chi}^{-1})\mid_{D})\ar[d]^{\mu_D}\\
 &  H^1 ( N_{D/S}  \otimes L_{\chi}^{-1})
}
\end{gather*}
where $\beta_D$ is induced by $\eta_D$, and $\gamma_D$ is induced by $\zeta_D$. The map $\mu_D$ is induced by the normal bundle sequence, and $\alpha_D$ is defined by the diagram.

\begin{definition}\label{dAlphaBeta}
Let $\chi\in G^{\vee}$, $\chi\neq 0$, be given. Let $\sA (\chi )$ be the set of those $D\in \fD$ such that $\langle \chi , \lambda (D)\rangle \neq p-1$ and $D$ is the strict transform of a line and, moreover, $\sigma^* \OO (H) \otimes L_{\chi}^{-1}$ is negative on $D$. 
Then we define the maps $\alpha_{\chi}, \beta_{\chi}$ as in the following diagram:
\begin{gather*}
\xymatrix{
\bigoplus_{\nu} H^0 (\OO_{E_{\nu}}(1) \otimes L_{\chi}^{-1} ) \ar[r]^{\quad\quad \delta_{\chi}} \ar[d]^{\beta_{\chi}=(\beta_D)_{D\in \sA (\chi )}} & H^1 (T_S\otimes L_{\chi}^{-1})\ar[d]^{\gamma_{\chi}=(\gamma_D)_{D \in \sA (\chi )}} \\
\bigoplus_{D\in \sA (\chi )}\bigoplus_{\nu} H^0 ( \OO_{E_{\nu}\cap D }(1)\otimes L_{\chi}^{-1} ) \ar[r]\ar[rd]_{\alpha_{\chi}=\oplus_{D \in \sA (\chi )} \alpha_D} & \bigoplus_{D\in \sA (\chi )} H^1 ( (T_S  \otimes L_{\chi}^{-1})\mid_{D})\ar[d]\\
 & \bigoplus_{D\in \sA (\chi )} H^1 ( N_{D/S}  \otimes L_{\chi}^{-1})
}
\end{gather*}
\end{definition}

\begin{theorem}\label{tDeformationsCoveringSurface}
Keeping the hypotheses of Theorem \ref{tStructureTheoremSpecial}, assume the following: 
\begin{enumerate}
\item[(a)]
For all $\chi\neq 0$ 
\[
\mathrm{reg}(\II_{\chi}) < d^{\chi}. 
\]
\item[(b)]
For all $D\in \fD$ and $\chi\in G^{\vee}$, $\chi\neq 0$, we have $D \cdot (D -L_{\chi}) < 0$.
\item[(c)]
For all $\chi\neq 0$, and for all $E_{\nu}$ the number of $D \in \sA (\chi )$ such that $D$ intersects $E_{\nu}$ is at least $2 - L_{\chi}.E_{\nu}$.
\end{enumerate}
Then the Kuranishi space $\mathrm{Def}_{(S, \fD)}$ is isomorphic to the Kuranishi space $\mathrm{Def}_{\tilde{S}}$. 
\end{theorem}

\begin{proof}
It follows from Corollary \ref{cCriterion} that we have to check the following conditions:
\begin{enumerate}
\item[(i)]
For all $\chi \in G^{\vee}$: $H^0 (S, T_S\otimes L_{\chi}^{-1}) =0$.
\item[(ii)]
For all $\chi\in G^{\vee}$, with $\chi\neq 0$ and all $D\in \fD$ with $\langle \chi , \lambda (D) \rangle   \neq p-1$: $$H^0 (D , \OO_{D }(D )\otimes L_{\chi}^{-1}))=0.$$
\item[(iii)]
For all $\chi \in G^{\vee}$, with $\chi\neq 0$, the natural map 
\[
H^1 (S, T_S\otimes L_{\chi}^{-1}) \to \bigoplus_{D\in \fD \colon \langle \chi , \lambda (D) \rangle \neq p-1} H^1 (D, \OO_{D}(D)\otimes L_{\chi}^{-1})
\]
induced by the natural maps $T_S\otimes L_{\chi}^{-1} \to \NN_{D\mid S} \otimes L_{\chi}^{-1}$ given by restricting germs of tangent vector fields to the normal bundle of $D$, is \emph{injective}. 
\end{enumerate}
We now proceed to deduce (i), (ii), (iii) from the assumptions of the Theorem. 

\medskip

\textbf{Step 1: Proving the vanishing in (i).} We have $H^0 (S, T_S) =0$ since we have blown up $p_1, \dots, p_4$. If $\chi \neq 0$, we use the exact  sequence (\ref{fTangentSequence})
 and the Euler sequence
\begin{gather}\label{fEuler}
0 \to \OO_{\P^2} \to 3\OO_{\P^2}(1) \to T_{\P^2} \to 0.
\end{gather}
This gives 
\begin{gather}\label{fEulerTensorLchi}
0\to L_{\chi}^{-1} \to 3 L_{\chi}^{-1}(\sigma^*H) \to \sigma^* T_{\P^2} \otimes L_{\chi}^{-1} \to 0.
\end{gather}
We have 
\[
h^0 (L_{\chi}^{-1}(\sigma^*H)) = h^2 (\II_{\chi} (d^{\chi}-1))
\]
and
\[
h^1 (L_{\chi}^{-1}) = h^1 (\II_{\chi} (d^{\chi}))
\]
by Corollary \ref{cLChi}. Now $H^2 (\II_{\chi} (d^{\chi}-1))=H^1 (\II_{\chi} (d^{\chi}))=0$ by assumption (a). Hence $H^0 (\sigma^* T_{\P^2} \otimes L_{\chi}^{-1})=0$, and thus $H^0 (T_S\otimes L_{\chi}^{-1}) =0$ for $\chi\neq 0$.

\medskip

\textbf{Step 2: Proving the vanishing in (ii).} Let $\chi\neq 0$ be given. Then we have $H^0 (D, \OO_{D}(D)\otimes L_{\chi}^{-1}))=0$ because each $D\simeq \P^1$ and $D\cdot (D-L_{\chi}) < 0$ for all $D\in\fD$ and $\chi\neq 0$ by (b). 

\medskip

\textbf{Step 3: Proving the injectivity of the map in (iii).} This is slightly more involved and we divide the proof into three sub-steps. 

\textbf{Step 3.1} We first show $\delta_{\chi}\colon \bigoplus_{\nu} H^0 (\OO_{E_{\nu}}(1) \otimes L_{\chi}^{-1} )\to  H^1 (T_S\otimes L_{\chi}^{-1})$ is an isomorphism.

By the sequence (\ref{fTangentSequence}), it is sufficient to show
$
h^i ( \sigma^* T_{\P^2} \otimes L_{\chi}^{-1}) = 0. 
$
for $i=0, 1$. 

We have already proved the vanishing of $h^0 ( \sigma^* T_{\P^2} \otimes L_{\chi}^{-1})$ in Step 1 using (\ref{fEulerTensorLchi}). The same argument shows  $h^{1}(L_{\chi}^{-1} (\sigma^*  H))=0$ and therefore we only need to prove that the map
\[
 H^2(L_{\chi}^{-1}) \to H^2(3 L_{\chi}^{-1}(\sigma^*H)) 
\]
is injective.

By Serre duality and Corollary  \ref{cLChi}, this is equivalent to the surjectivity of the map
\[
3H^0 (\II_{\chi}(d^{\chi} -1)) \to H^0 (\II_{\chi}(d^{\chi}))
\]
guaranteed by assumption (a) and Theorem \ref{tCastelnuovo}.

Hence to prove that the map in (iii) is injective it suffices to prove that both $\alpha_{\chi}$ and $\beta_{\chi}$ are injective. 

\medskip

\textbf{Step 3.2} We show that $\alpha_{\chi}$ is injective: It suffices to show that for a given $D\in \sA (\chi )$, the map
\[
\alpha_D\colon \bigoplus_{\nu} H^0 (\OO_{E_{\nu}\cap D} ) \to H^1 (N_{D/S}  \otimes L_{\chi}^{-1})
\]
is injective. 
Let $\sigma'=\sigma\mid_{\sigma^{-1}(L)}$.

Then we consider the diagram 
\[
\xymatrix{
	& 0 \ar[d]
	& 0 \ar[d]
	\\
	%0 \ar[r]
	& T_{D} \ar[r]^{\simeq} \ar[d]
	& (\sigma')^* T_{L}\mid_{D} \ar[d] %ar[r]
	%& 0 \ar[d]
	\\
	0 \ar[r] 
	&T_S|_{D} \ar[r] \ar[d]
	& \sigma^* T_{\P^2} |_{D} \ar[r] \ar[d]
	& \bigoplus_{\nu} \OO_{E_{\nu} \cap D} \ar[r] \ar[d]^{\simeq} 
	& 0
	\\
	0 \ar[r] 
	& N_{D/S} \ar[r] \ar[d]
	& (\sigma')^* N_{L/\P^2}\mid_{D} \ar[r] \ar[d]
	& \bigoplus_{\nu} \OO_{E_{\nu} \cap D} \ar[r] %\ar[d]
	& 0
	\\
	& 0
	& 0
	%& 0
	\\	
	}
\]
 The middle row is obtained by tensoring (\ref{fTangentSequence}) by $\OO_{D}$. The upper left commutative square is obtained because $\sigma'$ restricted to $D$ is an isomorphism onto $L$. Then the first column is the normal bundle sequence of $D$; the second column is the normal bundle sequence of $L$, pulled back via $\sigma'$ and restricted to $D$. The snake lemma then completes the diagram.

Tensoring the bottom row by $L_{\chi}^{-1}$, we find that $\bigoplus_{\nu} H^0 (\OO_{E_{\nu} \cap D})$ embeds into the space $H^1 ( (N_{D/S}\otimes L_{\chi}^{-1})\mid_{D})$ provided $$H^0 ( (\sigma')^* N_{L/\P^2}\mid_{D}\otimes L_{\chi}^{-1})= H^0 (L_{\chi}^{-1}(\sigma^* H)\mid_{D})=0.$$
The latter vanishing follows since $D \in \sA (\chi )$.

\medskip

\textbf{Step 3.3} We show that 
\[
\beta_{\chi}\colon \bigoplus_{\nu} H^0 (\OO_{E_{\nu}}(1) \otimes L_{\chi}^{-1} )\to \bigoplus_{D\in \sA (\chi )}\bigoplus_{\nu} H^0 ( \OO_{E_{\nu}\cap D}(1)\otimes L_{\chi}^{-1} )
\] 
is injective. Indeed, we show that for a given, but arbitrary $\nu$ 
\begin{gather}\label{fInjectivity}
H^0 (\OO_{E_{\nu}}(1) \otimes L_{\chi}^{-1} ) \to \bigoplus_{D\in \sA (\chi )} H^0 (\OO_{E_{\nu}\cap D}(1)\otimes L_{\chi}^{-1} )
\end{gather}
is injective. 

 Hence the left hand side in (\ref{fInjectivity}) will only be nonzero if either $E_{\nu}$ occurs with coefficient $0$ in $L_{\chi}$, in which case the left hand side is a two-dimensional vector space; or if $E_{\nu}$ occurs with coefficient $-1$, in which case the dimension of that vector space is $1$. By assumption (c), we have at least two $D$'s, $D \in \sA (\chi )$, intersecting $E_{\nu}$ in the first case, and at least one such $D$ in the second case. Hence the map in (\ref{fInjectivity}) is injective in both cases.
\end{proof}

%%%%%%%%%%%%%%%%%%%%%%%%%%%%%%%%%%%%%%%%%
\section{An example of a rigid, but not infinitesimally rigid incidence scheme 
              of points and lines in $\P^2$}
\label{sIncidenceSchemes}
%%%%%%%%%%%%%%%%%%%%%%%%%%%%%%%%%%%%%%%%%

In this Section we show how to construct an arrangement of lines $\mathfrak{L}$ having $q_1=(1:0:0), q_2=(0:1:0), q_3=(0:0:1), q_4=(1:1:1)$ as singular points, and such that the associated incidence scheme $\mathfrak{I}(\mathfrak{L})$ in the sense of Definition \ref{dLineArrIncScheme} is a double point. 

We will first consider certain generalised incidence schemes that we call {\it triangle schemes}. Their construction is very simple, with fixed data given by three points and three lines, and variable data also given by three points and three lines. We will show that under certain conditions the triangle scheme is a double point and we will choose a suitable triangle scheme $\fT^\heartsuit$ with this property.

Then we will produce from $\fT^\heartsuit$ a line arrangement  $\fL^\heartsuit$ with special properties, listed in Remark \ref{rabc}. Using these properties we prove that its associated incidence scheme  $\fI(\fL^\heartsuit)$ is isomorphic to  $\fT^\heartsuit$, and therefore a double point as well.

At the end of the section we will discuss these special properties, and how difficult is to find a triangle scheme such that the induced line arrangement has them.
 
In fact, this line arrangement does not seem to be at all unique: we believe that via a similar construction method, many other such arrangements can be found; but we confine ourselves to giving one particular example for the sake of definiteness and because already that needs quite a bit of work to construct.

\begin{definition}\label{dTriangleScheme}
A \emph{triangle scheme} $\mathfrak{T}(P,Q,R) \subset (\P^2)^3 \times ((\P^2)^{*})^3$ is a generalised incidence scheme defined as follows.

We take as associated fixed data the  lines $L_X, L_Y, L_Z$ through the coordinate points and three points  $P, Q, R \in \P^2$.

Then, denoting the variable points by $(X, Y, Z)\in (\P^2)^3$ and the variable lines by $(L_P, L_Q, L_R)\in ((\P^2)^*)^3$ we take incidence conditions
\begin{align*}
P&\in L_P&
Q&\in L_Q&
R&\in L_R&
X&\in L_X&
Y&\in L_Y&
Z&\in L_Z\\
X&\in L_P&
Y&\in L_R&
Z&\in L_Q&
X&\in L_Q&
Y&\in L_P&
Z&\in L_R\\
\end{align*}
\begin{figure}[h]
	\centering
	\includegraphics[scale=0.4]{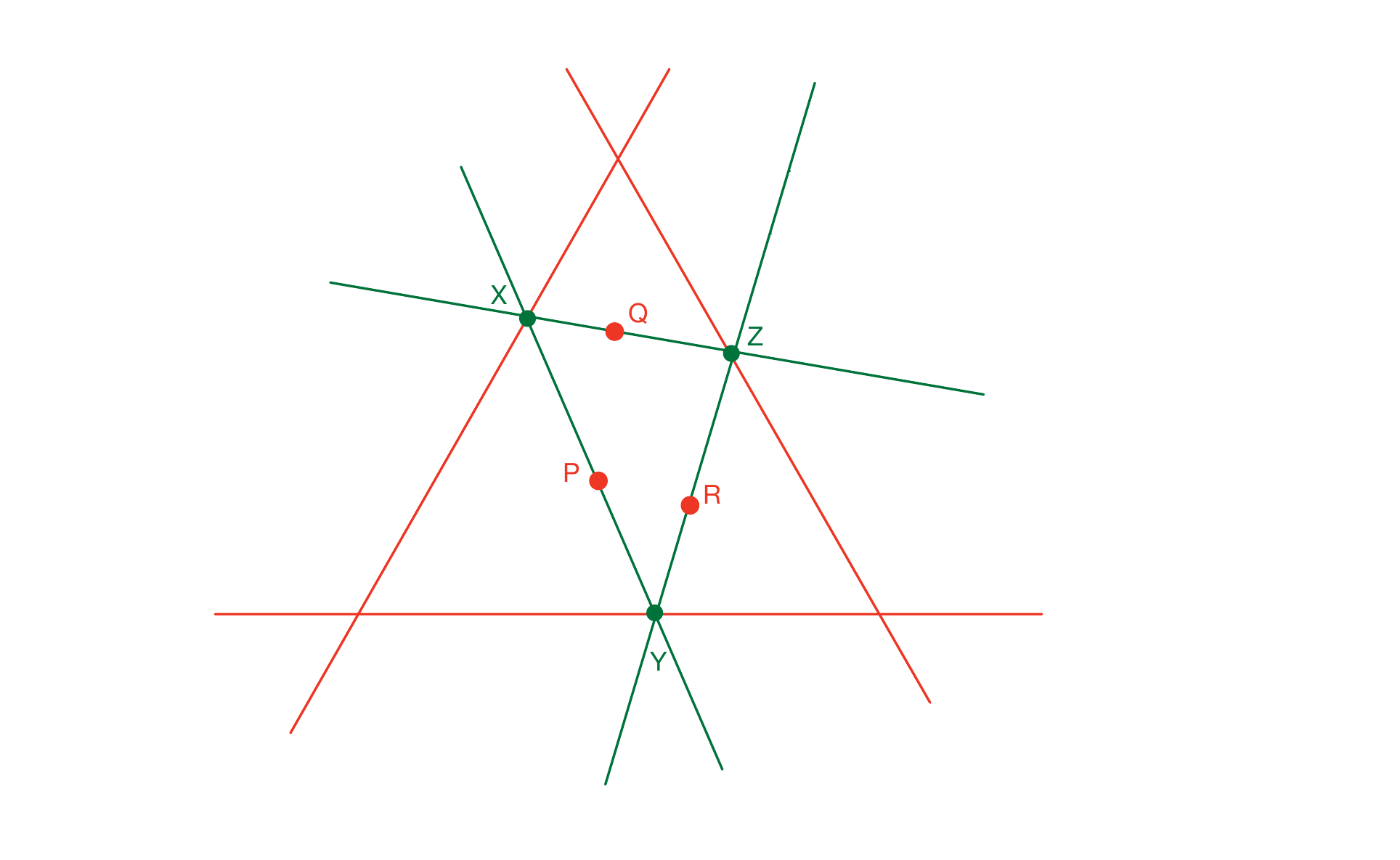}
	\caption{Description of $\mathfrak{T}(P,Q,R)$: the fixed data are red, the variable data are green}\label{fTriangle}
\end{figure}
\end{definition}

\begin{remark}\label{rNonredTriang}
For general choices of $P, Q, R$, the triangle scheme $\mathfrak{T}(P,Q,R)$ consists of two reduced points. 
This can be easily computed with Macaulay 2, as we did. 

However, we would like to give here some geometrical interpretation of it.
We note that  $Z$ is the image of $Y$ in $L_Z$  by the projection of center $R$. 
Similarly $X$ is the image of $Z$ in $L_X$ by the projection of center $Q$, and $Y$ is the image of $X$  in $L_Y$ by the projection of center $R$. 
So $Y$ is a fixed point for the projectivity of $L_Y$ obtained by composing these three projections, a projectivity that depends on the choice of the points $P,Q$ and $R$. 

For general choice of $P,Q,R$ we obtain a general projectivity of $L_Y \cong \P^1$, associated to a general invertible $2\times 2$ matrix, with two distinct eigenvalues and then exactly two fixed points. In this case the projection $ (\P^2)^3 \times ((\P^2)^{*})^3 \rightarrow \P^2$ corresponding to $Y$ (forgetting $X,Z,L_P,L_Q,L_R$) embeds  $\mathfrak{T}(P,Q,R)$ in $L_Y$, giving exactly the reduced scheme of the fixed points of this projectivity.

However for special choice of $P,Q$ and $R$ the matrix may have only one eigenvalue. 
For example, if we choose
\begin{gather*}
P=(1:1:2)\\
Q=(1:2:1)\\
R=(2:1:1)
\end{gather*}
then we get the projectivity $(a:0:b) \rightarrow (3b-2a:0:4b-3a)$. The matrix
$$
\begin{pmatrix}
-2 & -3\\
3&4\\
\end{pmatrix}
$$
has only one eigenvalue, $1$, with eigenspace of dimension $1$ giving a unique fixed point $Y=(1:0:-1)$. In this case in fact $\mathfrak{T}(P,Q,R)$ is a double point.
%See \cite{BBP-M2} for the computation.
\end{remark}

\begin{remark}
One can also see by intersection arguments that $\mathfrak{T}(P,Q,R)$ is of length $2$ when zero-dimensional: let $\P^1_X, \P^1_Y, \P^1_Z$ be the lines on which the points $X, Y, Z$ lie, and let $\P^1_P, \P^1_Q, \P^1_R$ be the pencils of lines through $P, Q, R$. Then the parameter space for $(X, Y, Z, L_P, L_Q, L_R)$ is 
\[
\P := \P^1_X \times \P^1_Y \times \P^1_Z \times \P^1_P \times \P^1_Q \times \P^1_R 
\]
Each point of $\P^1_X$ determines a divisor in $\P$, fibre of the point by the natural projection, whose class in $\Pic(\P)$ do not depend on the choice of the point: we denote it by $x$. Similarly we obtain classes $y$, $z$, $p$, $q$, $r$ by considering projections on the $\P^1$ labeled by the corresponding uppercase letter.

The locus of $6$-tuples $(X, Y, Z, L_P, L_Q, L_R)$ with $X\in L_P$ is a divisor on $\P$ of class $x+p$, and similarly for the other incidences.  
Thus the triangle scheme has class
\[
(x+p)(y+p)(x+q)(z+q)(y+r)(z+r) = 2 xyzpqr. 
\]
\end{remark}

\begin{remark}
Consider the discriminant
\[
	\Delta \subset \P^2_P \times \P^2_Q \times \P^2_R,
\]
the Zariski closure of the set of triples $P, Q, R$ such that $\fT(P,Q,R)$ is a double point. One can compute that $\Delta$ is
a divisor of multidegree $(2,2,2)$. If we write
\begin{align*}
P=&(P_0:P_1:P_2)&
Q=&(Q_0:Q_1:Q_2)&
R=&(R_0:R_1:R_2)
\end{align*}
then
\[
\Delta = (P_0Q_1R_2+P_1Q_2R_0+P_2Q_0R_1-P_2Q_1R_0)^2-4P_0P_1Q_0Q_2R_1R_2
\]
\end{remark}

The following is the triple we chose.

\begin{proposition}\label{pTriangleHeart}
The triangle scheme $\fT^\heartsuit := \fT(P,Q,R)$ with
\begin{align*}
	P &= (1:4:2) \\
	Q &= (3:14:3) \\ 
	R &= (14:25:1)
\end{align*}
is a non-reduced point isomorphic to $\mathrm{Spec}\, \C [t]/(t^2)$. 
\end{proposition}

\begin{proof}
This is checked in \cite{BBP-M2}. 
\end{proof}

This is the triangle scheme we are going to use. 
So, from now on, we will denote by $P,Q,R \in \P^2$ the points  whose coordinates are given in Proposition \ref{pTriangleHeart}.

We consider the two lines through $P$ defined by the polynomial
$$
(6x_0-4x_1+5x_2)
(6x_0-2x_1+x_2)
$$ the two lines through $Q$  defined by the polynomial
$$
(5x_0-3x_1+9x_2)
(x_0-3x_1+13x_2)
$$ and the two lines through $R$  defined by the polynomial
$$
(2x_0-x_1-3x_2)
(9x_0-5x_1-x_2)
$$
These six lines form a line arrangement $\fL'$ having $P,Q$ and $R$ as double points.

Consider the following iterative construction of point sets and line arrangements: 
\begin{enumerate}
\item Let $\fP_0 = \{q_1,q_2,q_3,q_4 \}$ 
\item Let $\fL_{i}$ be the set of lines that contain at least $2$ points of $\fP_{i-1}$.
\item Let $\fP_{i}$ be the set of points that lie on at least $2$ lines of $\fL_{i}$.
\end{enumerate}

$\fL_1$ contains $6$ lines, $\fL_3$ contains $25$ lines and $\fP_3$ contains $97$ points.

From this we construct our line arrangement, by adding $\fL'$  and the triangle
arrangement above:

\begin{definition}
Set 
\begin{align*}
	\fL^{+} :=& \fL_3 \cup 
	 \fL'&
	 \fL^\heartsuit &= \fL^+ \cup \{L_P,L_Q,L_R\}
\end{align*}
where
\[
	 \{L_P,L_Q,L_R\} := \fL(\fT^\heartsuit)
\]
is the line arrangement corresponding to the unique geometric point of $\fT^\heartsuit$. Furthermore let $\fP^+$ be the
set of intersection points of the line arrangement $\fL^+$.
\end{definition}

\begin{remark}\label{rabc}
The explicit coordinates of all $34$ lines in $\fL^\heartsuit$ are listed in Table \ref{tBuildingData} of Appendix \ref{Appendix}. 

One can check that
\begin{enumerate}
\item \label{iTwoFixedPoints} the $6$  lines  forming $\fL'$ each pass through at least $2$ points of $\fP_3$.
\item \label{iPQRintersections} two of them intersects in $P$, other two of them intersects in in $Q$ and the last two in $R$. 
\item \label{iNoFurtherIncidences} $L_P, L_Q$ and $L_R$ contain none of the points in $\fP^+$ except for $P$, $Q$ and $R$ respectively. 
\end{enumerate}
see \cite{BBP-M2} for a Macaulay2 script doing this computation.
\end{remark}

Our goal is to show that $\fI(\fL^\heartsuit) \cong \fT^\heartsuit$, For this we need the following

\begin{lemma}\label{lElimination}
Suppose that  $\mathfrak{I}$ is a generalised incidence scheme with associated fixed data $q_1, \dots , q_{m'}$ and $M_1, \dots , M_{n'}$. Assume that there is a variable line $L_k$ such that the incidence
\begin{gather}\label{fInc}
q_i\in L_k, \quad q_j\in L_k
\end{gather}
for $q_i\neq q_j$ is part of the defining set of equations. 

Then $\mathfrak{I}$ is isomorphic to the generalised incidence scheme $\mathfrak{I}'$ with associated fixed data $q_1, \dots , q_{m'}$ and $M_1, \dots , M_{n'}, M_{n'+1}$ where $M_{n'+1}$ is the unique line passing through the fixed points $q_i, q_j$, defined by the same set of  equations as $\mathfrak{I}$,  omitting the incidences in (\ref{fInc}) and replacing $L_k$ by $M_{n'+1}$ whenever it occurs. Hence $\mathfrak{I}'$ is a closed subscheme of $(\P^2)^m\times ((\P^2)^{*})^{n-1}$. We say that $\mathfrak{I}'$ becomes isomorphic to $\mathfrak{I}$ by eliminating $L_k$. 

A similar result holds for the elimination of a variable point that lies on two fixed lines. 
\end{lemma}

\begin{proof}
Let $q_i=(q_{i, 0}:q_{i, 1}:q_{i,2})$ and $q_j=(q_{j, 0}:q_{j, 1}:q_{j,2})$ with $q_{i,\mu}, q_{j,\mu}\in \C$. Moreover, let $L_k= (L_{k,0}:L_{k,1}:L_{k,2})$ where the $L_{k,\mu}$ is are variables. The incidences (\ref{fInc}) translate into the linear equations 
\[
\sum_{\mu} q_{i,\mu}L_{k,\mu} =0, \quad \sum_{\mu} q_{j,\mu}L_{k,\mu} =0.
\]
Since $q_i\neq q_j$ these define a reduced point in $(\P^2)^*$ whose coordinates are $M_{n'+1}$. Hence projecting $(\P^2)^m\times ((\P^2)^{*})^{n}$ onto $(\P^2)^m\times ((\P^2)^{*})^{n-1}$ by omitting the $k$-th copy of $(\P^2)^{*}$ gives an isomorphism between $\mathfrak{I}$ and $\mathfrak{I}'$. 
\end{proof}

\begin{proposition}\label{pIncNonReduced}
The incidence scheme $\fI(\fL^\heartsuit)$ is a double point.
\end{proposition}

\begin{proof}

We now show that $\fI(\fL^\heartsuit)$ is isomorphic to $\fT^\heartsuit$. 

First we recall that by Definition \ref{dLineArrIncScheme} the incidence scheme $\fI(\fL^\heartsuit)$ is the generalized incidence scheme associated to, as fixed data, the set of points $\fP_0 = \{q_1,q_2,q_3,q_4 \}$ and, as variable data, the set of all the lines of the arrangement and the set of all its singular points. In particular the variable data contain the lines $\{L_P,L_Q,L_R\}$ and the points $\{X,Y,Z\}$.

First we observe that $\fI(\fL^\heartsuit)$ is isomorphic to the generalised incidence scheme $\fI'$ obtained from $\fI(\fL^\heartsuit)$ by adding additional variable points for each point of $\fP^{+}$. 

By Lemma \ref{lElimination} we can ``eliminate'' all lines in $\fL_1$ since each contains two points belonging to the fixed data. So we move the lines in $\fL_1$ from the variable data to the fixed data obtaining an isomorphic incidence scheme. Then, since the lines in $\fL_1$ are now fixed, we can apply  Lemma \ref{lElimination} to eliminate the points of  $\fP_1$ (including those of  $\fP_0$ that were already in the fixed data). By the same argument we  eliminate successively the points and lines
$\fL_2, \fP_2, \fL_3,\fP_3$ from $\fI'$ using Lemma \ref{lElimination}.

Because of condition ({\ref{iTwoFixedPoints}}) in Remark \ref{rabc}. we can now eliminate the six lines in $\fL'$, and then all remaining points of $\fP^+$. 

The variable data of the resulting generalized  incidence scheme $\fI^\heartsuit$ is given exactly by the variable points $\{X,Y,Z\}$ and the variable lines $\{L_P,L_Q,L_R\}$ (the fixed data being given by all the lines and points that we have eliminated). 
It also counts the incidences defining $\fT^\heartsuit$ among its relations. 
It remains to check that there are no further incidence relations in $\fI^\heartsuit$. This follows from condition ({\ref{iNoFurtherIncidences}})  in Remark \ref{rabc}.
\end{proof}

\begin{remark}   \label{rHeightPQR}
Let's discuss the conditions in Remark \ref{rabc}, starting from conditions {(\ref{iTwoFixedPoints})} and {(\ref{iPQRintersections})}.
The lines in $\fL'$  are chosen in such a  way that the first two intersect in $P$, the second two intersect in $Q$ and the last 
two intersect in $R$. Furthermore each of these lines contains two points of $\fP_3$.
Such a situation is remarkably easy to arrange due to the following heuristic involving heights of point and lines in $\P^2$.  Here by the height of a rational point in $\P^2$ we mean the minimum absolute value of an entry in a set of integer homogeneous coordinates for the given point that are chosen such that their greatest common divisor is $1$. Similarly the height of a rational line in $\P^2$ is the height of the point in $(\P^2)^{\vee}$ representing the line. 

\medskip

First observe that the points of $\fP_3$ all have height at most $5$. Therefore a line $L$ containing two of the points of $\fP_3$ has height at most
$25$ (the coefficients of the equation of $L$ are determinants of $2 \times 2$-matrices in the coefficients of the points). Evaluating
such a line in a point $P$ of height $h$ gives a number between $-25h$ and $25h$. One of these numbers is $0$ so the probability of 
a point of height $h$ lying on one of these lines is approximately $1/(50h)$. Now we have about ${97 \choose 2} \simeq 5000$ 
such lines. So if the height of $P$ is less than $100$ we have a good chance of finding one (or sometimes $2$) lines 
passing though $P$ and $2$ points of $\fP_3$.

In other words, $\fP_4$ is so big that it already contains a considerable portion of the rational points of height at most $100$ in $\P^2$.
\end{remark}

\begin{remark}
The difficult condition for a random construction such as ours to satisfy is condition {(\ref{iNoFurtherIncidences})}. Here Remark \ref{rHeightPQR} works against us: Since the points of $\fP_3 \subset \fP^+$ have small height, the probability that one of them lies on $L_P$, $L_Q$ or $L_R$ is quite high. This probability is reduced if we choose $P$, $Q$ and $R$ of large height. If the height is too large we do not get the advantages described
in Remark \ref{rHeightPQR} so we are forced to choose $P$, $Q$ and $R$ of somewhat intermediate height (around $20$) as we have done above.
\end{remark}

\section{An example of a rigid, not infinitesimally rigid surface of general type}\label{sExampleSurface}

We have constructed a line arrangement $\fL^\heartsuit$ in the preceding Section whose associated incidence scheme $\mathfrak{I}(\fL^\heartsuit)$ is a nonreduced point. In this Section we show how to associate building data for an abelian cover to the initial datum of this line arrangement so that we can carry out the program set out in Section \ref{sIncidence}.

\begin{definition}
For our line arrangement $\fL^\heartsuit$ let
\[
	\lambda^\heartsuit \colon \fD \to (\Z/7)^4
\]
be the unique map satisfying the divisibility condition and with values $\lambda^\heartsuit (\bar{L_i})$, $i=1,\dots,33$ as depicted in Table \ref{tBuildingData} of Appendix \ref{Appendix}.  \end{definition}

\begin{lemma}
$\lambda^\heartsuit$ also satisfies the injectivity and the spanning condition.
\end{lemma}

\begin{proof}
This can be checked easily with a computer algebra program, for example with the Macaulay 2 script available at \cite{BBP-M2}.
\end{proof}

\begin{remark}
$\lambda^\heartsuit$ was found by a random search. We chose the values $\lambda(\bar{L}_i)$, $i=1,\dots,33$ such that they represented distinct points in $\P^4(\F_7)$. 
We then computed $\lambda(\bar{L}_{34})$ and $\lambda(E_i), \dots \lambda(E_{51})$ using the formulas in the proof of Lemma \ref{lLambdaUnique}. Finally we checked if all $85$ image points represented distinct points in $\P^4(\F_7)$. If not we started with new random values.

The chances of success of this scheme can be estimated by modelling the computed values as equally distributed random variables. Since 
\[
	|\P^4(\F_7)| = \frac{7^4-1}{6} = 400
\]
the chances of success are
\[
	\frac{400-33}{400} \cdot \ldots \cdot \frac{400-33-51}{400} \approx 0.000255 \approx \frac{1}{4000}
\]
 (this is a variant of the birthday-problem). Our computer took about 25 seconds for the search.

This was one of our criteria for choosing $G$. For smaller groups the birthday problem takes too long to solve. For larger groups we have to check too many $\chi$'s in what follows below.
 \end{remark}

\begin{corollary}\label{cCoverExists}
Then there exists a finite flat totally ramified Galois cover 
\[
	\pi \colon \tilde{S}^\heartsuit \to S
\]
with group $G = (\Z/7)^4$, branch locus $\sum_{D \in \fD} D$, and with the covering surface $\tilde{S}^\heartsuit$ smooth, with the property that 
\begin{align*}
& \pi_* \OO_{\tilde{S}^\heartsuit} = \bigoplus_{\chi \in G^{\vee}} L_{\chi}^{-1}.
\end{align*}
\end{corollary}

\begin{proof}
Apply Theorem \ref{tStructureTheoremSpecial} to $\lambda^\heartsuit$.
\end{proof}

\begin{theorem}\label{tMain1}
The  Kuranishi space of $\tilde{S}^\heartsuit$ is a non-reduced point.
\end{theorem}

\begin{proof}
By Proposition \ref{pIncNonReduced}, Proposition \ref{pIsoDefs} and Theorem \ref{tDeformationsCoveringSurface}, we only have to check that
\begin{enumerate}
\item[(a)]
For all $\chi\neq 0$ 
\[
\mathrm{reg}(\II_{\chi}) < d^{\chi}. 
\]
\item[(b)]
For all $D\in \fD$ and $\chi\in G^{\vee}$, $\chi\neq 0$, we have $D.(D -L_{\chi}) < 0$.
\item[(c)]
For all $\chi\neq 0$, and for all $E_{\nu}$ the number of $D \in \sA (\chi )$ such that $D$ intersects $E_{\nu}$ is at least $2 - L_{\chi}.E_{\nu}$.
\end{enumerate}
Note that (b) is automatically satisfied for $D$ an exceptional divisor, where we use that each exceptional divisor occurs with non-positive coefficient in $L_{\chi}$ by Lemma \ref{lCoeffExceptional}. 
We then check the remaining parts of all three conditions explicitly using the computer algebra system Macaulay2. See \cite{BBP-M2} for a script doing the computation.
\end{proof}

\begin{remark}
The conditions above always turned out to be fulfilled for random choices of building data for the group $(\Z/7)^4$. Conditions (a) and (b) were always satisfied by a wide margin, but for condition (c) larger $p$'s are better than smaller ones due to the following observation:

The  condition (c) is most restrictive for triple points $p_{\nu}$ and $\chi \in G^\vee$ such that the coefficient of $E_\nu$ in $L_\chi$ is zero. In this case we 
need that two of the three lines passing through $p_{\nu}$ are in $\sA(\chi)$. Since the coefficient of $E_\nu$ in $L_\chi$ is
\[
	 -\left\lfloor \frac{1}{p} \sum_{i| p_{\nu}\in L_i}  \llift \chi,\lambda(\bar{L}_i) \rlift \right\rfloor .
\]
the condition $\langle \chi, \lambda(\bar{L}_i) \rangle = p-1$ can be satisfied for at most one of the three lines passing through $p_\nu$ and in this case
$\langle \chi, \lambda(\bar(L)_i) \rangle = 0$ for the other two lines. If in this critical case $\sigma^* \OO(H) \otimes L_\chi^{-1}$ is positive for one of the other two lines, the condition (c) fails. We now compute how often this critical situation occurs:

Let $D_1,D_2,D_3 \in \fD$ be the strict transforms of the three lines passing through $p_\nu$. The critical situation occurs if 
\begin{align*}
	\langle \chi, \lambda(D_1) \rangle &= 0 \\
	\langle \chi, \lambda(D_2) \rangle &= 0 \\
	\langle \chi, \lambda(D_3) \rangle &= p-1 
\end{align*}
This is a linear system of equations in $(\Z/p)^r$ and the number of solutions is $p^{r-3}$, if solutions exist. Since any of the three lines can be the one with scalar product $p-1$ we have at most $3p^{r-3}$ critical $\chi$'s for each triple point. If we now want a fixed number of
about $400$ projective points in $\P^{r-1}(\Z/p)$ because of the birthday problem, we have
\[
	400 \approx p^{r-1}	\implies 3p^{r-3} \approx \frac{1200}{p^2}
\]
so the number of critical $\chi$'s is reduced quickly if we increase $p$. The reason we do not choose gigantic $p$'s, for example $G = (\Z/401)^2$
is that for fixed number of projective points the computation times rise linearly with $p$. Indeed the number of $\chi$'s for which we 
have to check the conditions above is
\[
	|G^\vee| =	p^r \approx 400p.
\]
\end{remark}

\section{Ampleness of the canonical class}\label{sAmple}

Now we show that the canonical class $K_{\tilde{S}^\heartsuit}$ is ample. We start by recalling the ramification formula for the behaviour of the canonical class under a covering. 

\begin{theorem}\label{tRamificationCanonical}
Let $\pi\colon \tilde{S} \to S$ be a smooth Galois cover with group $G=(\Z/p)^r$ obtained via building data as in Theorem \ref{tStructureTheoremSpecial}. 
Set $B =\sum_{g\in G} D_g$ for the branch divisor, $R=\pi^{-1}(B)$ for the ramification divisor (both effective and reduced).
Then
\[
K_{\tilde{S}} = \pi^*  K_S + (p-1) R
\]
Moreover $K_{\tilde{S}}$ is numerically equivalent, as ${\Q}-$divisor, to
\[
 \pi^* \left( K_S + \frac{p-1}{p} B \right).
\]
\end{theorem}

\begin{proof}
As is well known (cf. e.g. \cite[Anhang A.1, A.]{BHH87}), if $\pi\colon\tilde{S} \to S$ is a surjective holomomorphic mapping of smooth projective complex surfaces, then $K_{\tilde{S}}= \pi^* \, K_S + \bar{R}$ where $\bar{R}$ is the ramification or Jacobi-divisor of $\pi$ (given by the zeroes of the Jacobian determinant of the mapping $\pi$ in local coordinates on the base and covering). Now by construction, the cover is regularly ramified with constant local ramification order $p$: this means that if $y\in B$ and $x\in \pi^{-1}(y)$, then there are local coordinates $(u,v)$ centred at $y$, and local coordinates $(s, t)$ centred at $x$ such that: (a) if $y$ is a smooth point of $B$, then locally $B=(u=0)$ and $\pi^{-1}(B)=(s=0)$ (as sets), and $\pi (s,t)= (s^p, t)$; and if (b) $y$ is a double point of $B$, then $B=(uv=0)$, $\pi^{-1}(B) =(st=0)$, and $\pi (s, t) = (s^p, t^p)$. So we have $p$-fold cyclic ramification for every component of the ramification locus independent of the component.  This is because the cover is smooth and Galois, and for an irreducible component $R'$ of $\bar{R}$, its inertia subgroup in $G$ is cyclic, hence of order $p$ as $G$ is an elementary abelian $p$-group.

Thus denoting by $R$ the underlying reduced divisor of the ramification divisor,
\[
\pi^* (B) = p R, \quad (p-1)R = \bar{R}
\]
whence the assertion.
\end{proof}

\begin{proposition}\label{pKample}
With the notation of Definition \ref{dCoversFromLineArrangements} let $\mu_\nu$
be the number of lines $L_i$ of the line arrangement $\fL$ passing through $p_{\nu}$. Assume that there is a $\lambda$ satisfying the hypotheses of Theorem \ref{tStructureTheoremSpecial}. 
Assume furthermore that
\begin{enumerate}
\item
$p\ge 3$
\item
$[(p-1)n-3p]^2 - \sum_{\nu=1}^{m} [(p-1)\mu_\nu-(2p-1)]^2 >0$
\item
$\mu_\nu < \frac{(2p-1)n}{3p}$ for all $\nu=1, \dots , m$
\item
$n > \frac{3p}{p-1}$. 
\end{enumerate}
Then the canonical bundle $K_{\tilde{S}}$ is ample.
\end{proposition}

\begin{proof}
We use the fact that by Theorem \ref{tRamificationCanonical}
\[
pK_{\tilde{S}} = \pi^* (pK_S + (p-1)B)
\]
where $B\subset S$ is the branch divisor given by 
\[
B = \sum_{i=1}^{n} \bar{L}_i + \sum_{\nu = 1}^{m} E_{\nu}.
\]
Moreover,
\[
K_S = \sigma^* (K_{\P^2}) + \sum_{\nu =1}^{m} E_{\nu}. 
\]
We need to show that $\Delta:= pK_S + (p-1)B$ is ample (since $\pi$ is finite). We use the Nakai-Moishezon criterion for this, meaning we will show $\Delta^2 > 0$ and $\Delta \cdot \bar{C} > 0$ for every irreducible curve $\bar{C}$ on $S$. Now
\begin{gather}\label{fB}
B = n\, \sigma^* H - \sum_{\nu=1}^{m} (\mu_\nu-1)E_\nu \end{gather}
and hence
\begin{gather}\label{fDelta}
\Delta 
%= (2b-9) \sigma^*H - \sum_{\nu=1}^{m} \bigl(2(\mu_\nu-1)-3\bigr) E_\nu
= [(p-1)n-3p] \sigma^*H - \sum_{\nu=1}^{m} [(p-1)\mu_\nu-(2p-1)] E_\nu.\end{gather}
Since each $\mu_\nu\ge 3$
\[
(p-1)\mu_\nu-(2p-1) \ge p-2 > 0
\]
by a), hence
\begin{gather}\label{fTrivCheck}
\Delta\cdot E_\nu >0\; \forall\,\nu .
\end{gather}
Moreover, 
\begin{gather} 
\Delta^2 = [(p-1)n-3p]^2 - \sum_{\nu=1}^{m} [(p-1)\mu_\nu-(2p-1)]^2
%= 59^2 - 31\cdot 1^2 - 9 \cdot 3^2 - 5^2 - 9\cdot 7^2 -9^2 = 2822
>0
\end{gather}
by assumption b). 

\medskip

If now $\bar{C}$ is an irreducible curve on $S$ that is not contracted by $\sigma$, let $\sigma(\bar{C})=: C$ be its image in $\P^2$. We can write
\begin{gather}\label{fBarC}
\bar{C} = \sigma^* (C) - \sum_{\nu=1}^{m} s_\nu E_\nu
\end{gather}
where $s_\nu= \bar{C}\cdot E_\nu$ is the multiplicity with which $C$ passes through $p_\nu$. Let $d$ be the degree of $C$ in $\P^2$. We compute that
\begin{gather}\label{fInterDeltaBarC}
\Delta\cdot\bar{C}=  [(p-1)n-3p] d - \sum_{\nu=1}^{m} [(p-1)\mu_\nu-(2p-1)] s_\nu .
\end{gather}
We need to show that this is bigger than zero, which will finish the proof. For this consider the $n\times m$ matrix (whose rows we imagine to be indexed by the lines, and columns indexed by the points we blow up) with entries $a_{i \nu}$ defined by 
\begin{gather*}
a_{i \nu} = \begin{cases} 0 &\mbox{if} \: p_{\nu}\notin L_i \\ 
s_{\nu} & \mbox{if}\:  p_{\nu} \in L_i . \end{cases} 
\end{gather*}
Then we have
\begin{gather}\label{fMatrix}
\forall\, i \colon \sum_{\nu=1}^{m} a_{i\nu} \le d,  \\
\forall\, \nu \colon 
\sum_{i=1}^{n} a_{i \nu} = \mu_\nu s_\nu
\end{gather}
(the first inequality because of Bezout, the latter equalities from the definition). Hence we can estimate
\begin{gather}\label{fDoubleSum}
\sum_{i,\nu} a_{i\nu} = \sum_{\nu=1}^{m} \mu_\nu s_\nu  \le n\cdot d .
\end{gather}
In other words,
\begin{gather}\label{fEstimateDegree}
d \ge \frac{1}{n} \sum_{\nu=1}^{m} \mu_\nu s_{\nu}  .
\end{gather}
Using formula (\ref{fInterDeltaBarC}), we can use this to estimate $\Delta\cdot\bar{C}$:
\begin{gather}\label{fConclusion}
\Delta\cdot \bar{C} \ge \frac{[(p-1)n-3p]}{n}\cdot  \sum_{\nu=1}^{m} \mu_\nu s_\nu
- \sum_{\nu=1}^{m} [(p-1)\mu_\nu-(2p-1)] s_\nu 
  \\
= \sum_{\nu=1}^{m} \Bigl( (2p-1) - \frac{3p\, \mu_\nu}{n}\Bigr)s_\nu. \nonumber
\end{gather}
This is clearly strictly greater than zero if for all $\nu$ 
\[
	\mu_\nu < \frac{(2p-1)n}{3p}
\]
(assumption c)) and one of the $s_\nu$ is nonzero. If all $s_\nu$ are zero, $\Delta\cdot \bar{C} > 0$ by formula (\ref{fInterDeltaBarC}) since in that case
\[
\Delta\cdot \bar{C}= [(p-1)n-3p] d > 0
\]
by assumption d).
\end{proof}

\begin{theorem}\label{tMain2}
The canonical class of the surface $\tilde{S}^\heartsuit$ is ample. In particular, $\tilde{S}^\heartsuit$ is of general type and coincides with its canonical model which is rigid, but not infinitesimally rigid.
\end{theorem}

\begin{proof}
It suffices to check that the hypotheses of Proposition \ref{pKample} are satisfied for $\fL^\heartsuit$ and $p=7$. This is done in \cite{BBP-M2} using Macaulay2. 
\end{proof}

As a sanity check and to locate our $\tilde{S}^\heartsuit$ in the geography of the surfaces of general type, we want to compute $K^2_{\tilde{S}^\heartsuit}$ and $\chi (\tilde{S}^\heartsuit)$. 

\begin{proposition}\label{pKSquareChi}
We keep the notation of Definition \ref{dCoversFromLineArrangements} and assume that we are given a $\lambda$ satisfying the hypotheses of Theorem \ref{tStructureTheoremSpecial}. Then:
\begin{gather*}
K^2_{\tilde{S}} = \bigl( pK_S + (p-1)B \bigr)^2 \cdot p^{r-2}
\end{gather*}
and
\begin{gather*}
\chi (\OO_{\tilde{S}}) = \sum_{\chi \in G^{\vee}} \chi (L_{\chi}^{-1}) 
\end{gather*}
where
\[
\chi (L_{\chi}^{-1}) = \frac{1}{2} L_{\chi}^{-1} .(L_{\chi}^{-1} -K_S) + 1 .
\]
Moreover,
\[
p_g(\tilde{S}) =\sum_{\chi\in G^{\vee}} h^0 (S, K_S \otimes L_{\chi}).
\]
\end{proposition}

\begin{proof}
By Theorem \ref{tRamificationCanonical}
\[
pK_{\tilde{S}} = \pi^* (pK_S + (p-1)B)
\]
hence
\[
p^2 K_{\tilde{S}}^2 = \deg (\pi ) (pK_S + (p-1)B)^2 = p^r  (pK_S + (p-1)B)^2
\]
which implies the first formula. The second formula follows since by Theorem \ref{tStructureTheoremSpecial}
\[
\pi_* \OO_{\tilde{S}} = \bigoplus_{\chi \in G^{\vee}} L_{\chi}^{-1}. 
\]
The formula for $p_g$ follows from \cite[Prop. 4.1 (c)]{Par91}.
\end{proof}

\begin{theorem}\label{tKSquareChi}
We have
\[
K^2_{\tilde{S}^\heartsuit}= 1,260,966 , \quad \chi (\tilde{S}^\heartsuit) = 151,851, \quad q (\tilde{S}^\heartsuit)=0 
\]
Therefore, 
\[
K^2_{\tilde{S}^\heartsuit}/\chi (\tilde{S}^\heartsuit) \simeq 8.3.
\]
\end{theorem}

\begin{proof}
We compute $K^2_{\tilde{S}^\heartsuit}$,  $\chi (\tilde{S}^\heartsuit)$, $p_g(\tilde{S}^\heartsuit)$ and $q (\tilde{S}^\heartsuit)$ in \cite{BBP-M2}
applying Proposition \ref{pKSquareChi} together with Corollary \ref{cLChi} and $q = 1- \chi + p_g$.

We give a direct proof of the vanishing of $q (\tilde{S}^\heartsuit)$. We have already checked that, for all $\chi\neq 0$, $\mathrm{reg}(\II_{\chi}) < d^{\chi}$,
that implies by Corollary \ref{cLChi} and Serre duality that all $h^1(S,L_{\chi}^{-1})$ vanish. 
Since $\pi$ is finite, $q (\tilde{S}^\heartsuit)=h^1(S,\pi_* \OO_{\tilde{S}})= \sum_{\chi \in G^{\vee}} h^1(S,L_{\chi}^{-1})=0$. 
\end{proof}
\begin{remark}\label{rInequalities}
Note that for a minimal surface $X$ of general type over $\C$, the dimension of the Kuranishi space of $X$ is bounded below by $10\chi (X) - 2 K_X^2$, see \cite{Ku65}. Moreover, any such surface satisfies $K^2_X \le 9\chi (X)$ by the Bogomolov-Miyaoka-Yau inequality. 
\end{remark}

%\newpage 

\appendix

\section{}\label{Appendix}

\begin{table}[h!]
\begin{center} 
\begin{tabular}{|c|c|c|}
\hline
$i$ & $L_i$ & $\lambda \Bigl(\bar{L}_i\Bigr)$ \\
\hline
$1$ & $\begin{pmatrix}
0&0&1\end{pmatrix}$&$\begin{pmatrix}
2&4&3&5\end{pmatrix}$\\
$2$ &$\begin{pmatrix}
0&1&0\end{pmatrix}$&$\begin{pmatrix}
6&2&4&2\end{pmatrix}$\\
$3$ &$\begin{pmatrix}
1&0&0\end{pmatrix}$&$\begin{pmatrix}
1&2&4&1\end{pmatrix}$\\
$4$ &$\begin{pmatrix}
0&1&{-1}\end{pmatrix}$&$\begin{pmatrix}
6&5&4&2\end{pmatrix}$\\
$5$ &$\begin{pmatrix}
1&0&{-1}\end{pmatrix}$&$\begin{pmatrix}
5&1&2&6\end{pmatrix}$\\
$6$ &$\begin{pmatrix}
1&1&{-1}\end{pmatrix}$&$\begin{pmatrix}
3&0&4&3\end{pmatrix}$\\
$7$ &$\begin{pmatrix}
1&{-1}&0\end{pmatrix}$&$\begin{pmatrix}
3&2&1&0\end{pmatrix}$\\
$8$ &$\begin{pmatrix}
1&1&0\end{pmatrix}$&$\begin{pmatrix}
4&3&4&3\end{pmatrix}$\\
$9$ &$\begin{pmatrix}
1&1&{-2}\end{pmatrix}$&$\begin{pmatrix}
1&6&5&6\end{pmatrix}$\\
$10$ &$\begin{pmatrix}
1&{-1}&1\end{pmatrix}$&$\begin{pmatrix}
5&1&1&3\end{pmatrix}$\\
$11$ &$\begin{pmatrix}
1&{-1}&{-1}\end{pmatrix}$&$\begin{pmatrix}
2&1&6&1\end{pmatrix}$\\
$12$ &$\begin{pmatrix}
0&2&{-1}\end{pmatrix}$&$\begin{pmatrix}
2&2&0&3\end{pmatrix}$ \\
$13$ &$\begin{pmatrix}
2&0&{-1}\end{pmatrix}$&$\begin{pmatrix}
6&1&5&6\end{pmatrix}$\\
$14$ &$\begin{pmatrix}
1&0&1\end{pmatrix}$&$\begin{pmatrix}
3&4&2&0\end{pmatrix}$\\
$15$ &$\begin{pmatrix}
1&{-2}&1\end{pmatrix}$&$\begin{pmatrix}
6&6&3&2\end{pmatrix}$\\
$16$ &$\begin{pmatrix}
1&1&1\end{pmatrix}$&$\begin{pmatrix}
3&6&6&6\end{pmatrix}$\\
$17$ &$\begin{pmatrix}
1&{-3}&1\end{pmatrix}$&$\begin{pmatrix}
3&6&4&6\end{pmatrix}$\\
\hline
\end{tabular}
\quad
\begin{tabular}{|c|c|c|}
\hline
$i$ & $L_i$ & $\lambda \Bigl(\bar{L}_i\Bigr)$ \\
\hline
$18$ & $\begin{pmatrix}
0&1&{-2}\end{pmatrix}$&$\begin{pmatrix}
4&2&4&2\end{pmatrix}$\\
$19$ &$\begin{pmatrix}
2&{-1}&0\end{pmatrix}$&$\begin{pmatrix}
3&5&2&6\end{pmatrix}$\\
$20$ &$\begin{pmatrix}
1&1&{-3}\end{pmatrix}$&$\begin{pmatrix}
6&4&4&6\end{pmatrix}$\\
$21$ &$\begin{pmatrix}
3&{-1}&{-1}\end{pmatrix}$&$\begin{pmatrix}
2&4&0&2\end{pmatrix}$\\
$22$ &$\begin{pmatrix}
0&1&1\end{pmatrix}$&$\begin{pmatrix}
1&3&4&5\end{pmatrix}$\\
$23$ &$\begin{pmatrix}
2&{-1}&{-1}\end{pmatrix}$&$\begin{pmatrix}
3&1&4&5\end{pmatrix}$ \\
$24$ &$\begin{pmatrix}
1&0&{-2}\end{pmatrix}$&$\begin{pmatrix}
3&4&2&6\end{pmatrix}$\\
$25$ &$\begin{pmatrix}
1&{-2}&0\end{pmatrix}$&$\begin{pmatrix}
3&6&3&3\end{pmatrix}$\\
$26$ &$\begin{pmatrix}
6&{-4}&5\end{pmatrix}$&$\begin{pmatrix}
5&4&6&2\end{pmatrix}$\\
$27$ &$\begin{pmatrix}
6&{-2}&1\end{pmatrix}$&$\begin{pmatrix}
3&2&1&1\end{pmatrix}$\\
$28$ &$\begin{pmatrix}
5&{-3}&9\end{pmatrix}$&$\begin{pmatrix}
5&5&4&6\end{pmatrix}$\\
$29$ &$\begin{pmatrix}
1&{-3}&13\end{pmatrix}$&$\begin{pmatrix}
6&5&6&4\end{pmatrix}$\\
$30$ &$\begin{pmatrix}
2&{-1}&{-3}\end{pmatrix}$&$\begin{pmatrix}
2&6&3&1\end{pmatrix}$\\
$31$ &$\begin{pmatrix}
9&{-5}&{-1}\end{pmatrix}$&$\begin{pmatrix}
6&4&0&4\end{pmatrix}$\\
$32$ &$\begin{pmatrix}
8&9&{-22}\end{pmatrix}$&$\begin{pmatrix}
3&5&2&0\end{pmatrix}$\\
$33$ &$\begin{pmatrix}
20&{-9}&22\end{pmatrix}$&$\begin{pmatrix}
5&2&5&0\end{pmatrix}$\\
$34$ &$\begin{pmatrix}
20&{-9}&{-55}\end{pmatrix}$&$\begin{pmatrix}
5&5&4&4\end{pmatrix}$\\
\hline
\end{tabular}
\end{center}
\caption{}\label{tBuildingData}
\end{table}

The following examples contain some computations done by hand which we use to test the correctness of the computer code in \cite{BBP-M2}. 

\begin{example} \label{eCoeff}
The coordinate point $(1:0:0)$ lies on $6$ of the above lines, i.e. those with $0$ in the first component. Let $E$ be the exceptional divisor on $S$ over $(1:0:0)$. Consider now 
\[
	\chi = (0,0,0,1) \in \bigl((\Z/ 7)^4)^{\vee}.
\]
We have 
\[
	\sum_{i|\bar{L}_i \cap E \not=0} \langle \chi, \lambda (\bar{L}_i)\rangle = 5+2+2+3+2+5 = 19
\]
which is the sum of the last entries of the labels corresponding to the above $6$ lines.
It follows that the coefficient of $E$ in $L_\chi$ is
\[
	- \lfloor 19 / 7 \rfloor = -2
\]
Furthermore we have
\[
	\sum_{i=1}^{34} \langle \chi, \lambda (\bar{L}_i) \rangle = 
	%5+2+1+2+6+3+0+3+6+3+1+3+6+0+2+6+6+2+6+6+2+5+5+6+3+2+1+6+4+1+4+0+0+4 =
	112
\]
which is the sum of the last entries of all labels in the above table. Notice that this is divisible by $7$ and
the coefficient of $\sigma^* H$ in $L_\chi$ is
\[
	112/7 = 16.
\]

For $\chi' = 2 \chi = (0,0,0,2)$ we have
\[
	\sum_{i|\bar{L}_i \cap E \not=0} \langle \chi', \lambda (\bar{L}_i)\rangle = \sum_{i|\bar{L}_i \cap E \not=0} \langle 2\chi, \lambda (\bar{L}_i)\rangle =  3+4+4+6+4+3 = 24
\]
and the coefficient of $E$ in $L_{\chi'}$ is $-3$.
\end{example}

\begin{example}
Consider the point $(2:1:0)$. It lies on three lines, namely the ones with dual coordinates $(0:0:1)$, $(1:-2:1)$ and $(1:-2:1)$. We want to find those $L_\chi$ where
\begin{enumerate}
\item $E_{(2:1:0)}$ has coefficient $0$ in $L_\chi$
\item $\sA (\chi )$ contains only the first $2$ lines.
\end{enumerate}
This can happen only for those $\chi$ whose scalar product with the labels of the first $2$ lines is $0$ and $p-1$ with
the label of the third line. This gives a linear system of equations
\[
	\begin{pmatrix}
	  2 & 4 & 2 & 5 \\
	  6 & 6 & 3 & 2 \\
	  3 & 6 & 3 & 3
	 \end{pmatrix} 
	 \cdot \chi
	 = \begin{pmatrix} 0 \\ 0 \\ p-1 \end{pmatrix}
\]
over $\Z/7$. The solutions are
\[
	(5,4,3,0)^t + k  (5,2,4,1)^t  \quad \text{ with $k \in \Z/7$.}
\]
	
\end{example}

\begin{example}
Consider $\chi = (0,0,0,1)$ and the line $L_{26}$ with dual coordinate $(6:-4:5)$. This line contains only two blowup points, namely $(1:-1:-2)$ and $(1:4:2)$. The coefficients of the corresponding exceptional divisors in $L_\chi$ are both $-1$. The coefficient of $\sigma^* H$ in $L_\chi$ is $16$ as computed in Example \ref{eCoeff}. 

It follows that
\[
	\bar{L}_{26} . (\sigma^* H-L_\chi) = 1\cdot (1-16)H^2 - E_{(1:-1:-2)}^2 - E_{(1:4:2)}^2 = -13.
\]
\end{example}

\

\emph{On behalf of all authors, the corresponding author states that there is no conflict of interest. }

\providecommand{\bysame}{\leavevmode\hbox to3em{\hrulefill}\thinspace}
\providecommand{\href}[2]{#2}

\end{document}